\documentclass{article}
\usepackage{amsfonts}
\usepackage{amsmath,amssymb}
\usepackage[dvips]{graphicx}
\usepackage{algorithm}
\usepackage{algorithmicx}
\usepackage{algpseudocode}
\usepackage{yhmath}
\usepackage{qtree}
\usepackage{xcolor}
\usepackage{epsfig}
\usepackage{lineno,hyperref}
\usepackage{amsmath}
\usepackage{tabularx}
\usepackage{calc}
\usepackage{makecell}
\usepackage{multirow}
\usepackage{amsthm}
\usepackage{mathtools}
\usepackage{bbm}
\usepackage{bm}
\usepackage[titletoc]{appendix}
\usepackage{makecell}
\usepackage{amssymb}
\usepackage{amsfonts}
\usepackage{mathrsfs}
\usepackage{indentfirst}
\usepackage{cases}
\usepackage{lipsum}
\usepackage{url}
\usepackage[titletoc]{appendix}
\usepackage{hyperref}
\usepackage{subfigure}
\usepackage{lineno,hyperref}
\usepackage{amsmath}
\usepackage{verbatim}
\usepackage{graphicx}
\usepackage{geometry}
\usepackage{tabu}

\usepackage[T1]{fontenc}
\usepackage[utf8]{inputenc}

\usepackage{booktabs}

\usepackage{lipsum}

\newcommand{\be}{\begin{equation}}
\newcommand{\ee}{\end{equation}}
\newcommand{\ba}{\begin{array}}
\newcommand{\ea}{\end{array}}
\newcommand{\bea}{\begin{eqnarray*}}
\newcommand{\eea}{\end{eqnarray*}}
\newcommand{\bean}{\begin{eqnarray}}
\newcommand{\eean}{\end{eqnarray}}

\newtheorem{theorem}{Theorem}[section]
\newtheorem{lemma}{Lemma}[section]
\newtheorem{remark}{Remark}[section]
\newtheorem{proposition}{Proposition}[section]
\newtheorem{definition}{Definition}[section]

\usepackage{mathtools}
\usepackage{mdframed}
\theoremstyle{remark}

\newcommand{\lmref}[1]{Lemma \ref{#1}}

\newcommand{\defref}[1]{Definition \ref{#1}}

\renewcommand{\raggedright}{\leftskip=0pt \rightskip=0pt plus 0cm}

\theoremstyle{definition}

\newcounter{proofc}
\renewcommand\theproofc{(\arabic{proofc})}
\DeclareRobustCommand\stepproofc{\refstepcounter{proofc}\theproofc}

\usepackage[noblocks]{authblk}

\begin{document}

\title{Sharp Decay Estimates for the Vlasov-Poisson and Vlasov-Yukawa Systems with Small Data}
\author[1]{Xianglong Duan}
\affil[1]{\small Laboratoire de Math\'ematiques d'Orsay, Univ Paris-Sud 11, CNRS, Universit\'e Paris-Saclay, F-91405, Orsay, France (\href{mailto:xianglong.duan@math.u-psud.fr}{xianglong.duan@math.u-psud.fr})}
\date{}
\maketitle

\begin{abstract}
In this paper, we present sharp decay estimates for small data solutions of the following two systems: the Vlasov-Poisson (V-P) system in dimension 3 or higher and the Vlasov-Yukawa (V-Y) system in dimension 2 or higher. We rely on a modification of the vector field method for transport equation as developed by Smulevici in 2016 for the Vlasov-Poisson system. Using the Green's function in particular to estimate bilinear terms, we improve Smulevici's result by requiring only $L^1_{x,v}$ bounds for the initial data and its derivatives. We also extend the result to the Vlasov-Yukawa system. 
\end{abstract}

\tableofcontents


\section{Introduction}


In this paper, we mainly study the following Vlasov-Poisson (V-P) system for $n\geq 3$,
\begin{gather}
\partial_t f + v\cdot\nabla_x f + \mu \nabla_x \phi \cdot\nabla_v f=0,\quad x,v\in\mathbb{R}^n,\label{eq.vlasov}\\
\triangle \phi=\rho (f), \label{eq.poisson}\\  
f(t=0)=f_0(x,v).
\end{gather}
where $\mu=\pm 1$ corresponding to an attractive or repulsive force and $\rho(f)$ is given by
$$\rho(f)(t,x)=\int_{\mathbb{R}^n} f(t,x,v){\rm d}v.$$
The function $f$ represents the density of particles and is therefore non-negative initially\footnote{In the small data regime considered in this paper, the sign of $f$ will be irrelevant. }. 
With a slight modification of the proof, our method also applies to the Vlasov-Yukawa\footnote{See \cite{Yukawa} for the original derivation of the Yukawa interaction.} (V-Y) system in dimension $n\geq2$
\begin{gather}
\partial_t f + v\cdot\nabla_x f + \mu \nabla_x \phi \cdot\nabla_v f=0,\quad x,v\in\mathbb{R}^n,\\
\triangle \phi -m^2\phi=\rho (f),\label{eq.screened.poisson}\\  
f(t=0)=f_0(x,v).
\end{gather}
where the Poisson equation \eqref{eq.poisson} is replaced by the screened Poisson equation \eqref{eq.screened.poisson} as a short-range correction.	Here $m>0$ represents the mass of particles which is assumed to be a positive constant. Without loss of generality, we will assume $m=1$ throughout  this paper. In both of the two systems, V-P and V-Y, the function $\phi(t,x)$ which solves either the Poisson equation \eqref{eq.poisson} or the screened Poisson equation \eqref{eq.screened.poisson} can be expressed explicitly in the form
$$\phi(t,x) = G_m*\rho(f)=\int_{\mathbb{R}^n}G_m(x-y)\rho(f)(t,y){\rm d}y,$$
where $G_m$ is the Green's function with respect to the operator $\triangle$ ($m=0$) or $\triangle -\mathrm{Id}$ ($m=1$). More precisely, for the V-P system in dimension $n\geq 3$, $G_0$ has the form
$$G_0(x)=\frac{c_n}{|x|^{n-2}},$$ 
where $c_n$ is a constant depending only on the dimension $n$. For the V-Y system, $G_1(x)$ has the form
$$G_1(x)=\frac{c'_n}{|x|^{\frac{n}{2}-1}}K_{\frac{n}{2}-1}(|x|),$$
where $c'_n$ is another constant and $K_{\nu}(r)$ is the modified Bessel function of the second kind, which has an integral expression of the form, when $r>0$ (see p. 181, \cite{Bessel}),
$$K_{\nu}(r)=\int_0^{\infty}e^{-r\cosh \lambda} \cosh (\nu \lambda){\rm d}\lambda.$$ 


The aim of this paper is to derive sharp asymptotics on the solutions under a small data assumption but optimal decay on the initial data. More precisely, our assumptions on the initial data implies that $\int_x \int_v f( t=0, x, v){\rm d}x {\rm d}v$ is finite, but we do not need more decay in $(x,v)$ initially. Similar assumptions are made for derivatives of $f$. 

The main analytic novelty of the paper relies on using the explicit expression of $\phi(t,x)$ thanks to the Green's function formula. Using this explicit representation, we can prove bilinear estimates for quantities of the form $|| \nabla_x Z^\gamma(\phi)Z^\beta(f)||_{L^1_{x,v}}$, where $Z^\gamma$ and $Z^\beta$ are differential operators of order $|\gamma|$ and $|\beta|$. Compared to for instance \cite{Smulevici.V-P}, we then avoid the use of the Calder\'on-Zygmund inequality. In particular, this allows us to avoid the use of $v$-weighted $L^p$ norms, which in \cite{Smulevici.V-P} appeared due to the failure of the Calder\'on-Zygmund inequality in $L^1$. A second important difference compared to \cite{Smulevici.V-P} is a different treatment of the higher order estimates in the low dimensional case. Due to the use of modified vector field, one typically encounters a loss of derivative in the estimates. In \cite{Smulevici.V-P}, this loss was recovered using the elliptic nature of the Poisson equation and the higher order estimates relied again on the Calder\'on-Zygmund inequality. Here, we write and exploit the higher order commutation formula differently to avoid the loss of derivatives, see Lemma \ref{lm.mv.YY}. 


\subsection{The Main Results}
Our main theorems can then be stated as follows.

\begin{theorem}[High Dimensional Case]\label{thm.high}
	Let $n \geq 4$ in the V-P case and $n \ge 3$ in the V-Y case. Let $N \geq 2n$. Consider initial data $f_0: \mathbb{R}^n_x \times \mathbb{R}^n_v$ satisfying 
	\begin{equation}
		\mathcal{E}_N[f_0]:= \sum_{|\alpha|\leq N, Z^{\alpha}\in\gamma^{|\alpha|}} \|Z^{\alpha}f_0\|_{L^1(\mathbb{R}^n_x\times\mathbb{R}^n_v)} \leq \epsilon,
	\end{equation}
	where $Z^\alpha$ is a combination of $|\alpha|$ vector fields in $\gamma^{|\alpha|}$ as defined in \defref{def.vector.fields}. Then, there exists $\epsilon_0>0$, such that if $\epsilon \le \epsilon_0$, then the solution $(f,\phi)$ arising from $f_0$ is globally defined and satisfies the uniform bounds
	
	\begin{equation}\label{hd}
	\mathcal{E}_N[f(t)]\leq 2\epsilon.
	\end{equation}
	
\end{theorem}

\begin{theorem}[Low dimensional Case]\label{thm.low}
	Let $n=3$ in the V-P case and $n=2$ in the V-Y case. 
	Let $N \geq 2n+3$. Consider initial data $f_0: \mathbb{R}^n_x \times \mathbb{R}^n_v$ satisfying 
	\begin{equation}
	\mathcal{E}_N[f_0]:= \sum_{|\alpha|\leq N, Y^{\alpha}\in\gamma_m^{|\alpha|}} \|Y^{\alpha}f_0\|_{L^1(\mathbb{R}^n_x\times\mathbb{R}^n_v)} \leq \epsilon,
	\end{equation}
	where $Y^\alpha$ is a combination of $|\alpha|$ modified vector fields in $\gamma_m^{|\alpha|}$  as defined in \defref{def.modified}. Then, there exists $\epsilon_0>0$, such that if $\epsilon \leq \epsilon_0$, then the solution $(f,\phi)$ arising from $f_0$ is globally defined and satisfies the uniform bounds
	\begin{equation}\label{lm.d}
	\mathcal{E}_N [f(t)] \leq 2\epsilon. 
	\end{equation}
\end{theorem}
\begin{remark}
	In the work \cite{Smulevici.V-P}, a smilar result was obtained but the norms considered had additional $v$-weighted $L^p$-norms. More precisely, in \cite{Smulevici.V-P}, the high dimensional energy, denoted by $\mathcal{E}_{N,\delta}[f]$ for any $\delta>0$, is defined as
		$$\mathcal{E}_{N,\delta}[f]:= \sum_{|\alpha|\leq N,  Z^{\alpha}\in\gamma^{|\alpha|}} \|Z^{\alpha}f\|_{L^1(\mathbb{R}^n_x\times\mathbb{R}^n_v)} + \sum_{|\alpha|\leq N,  Z^{\alpha}\in\gamma^{|\alpha|}}\|(1+|v|^2)^{\frac{\delta(\delta+n)}{2(1+\delta)}}Z^{\alpha}f\|_{L^{1+\delta}(\mathbb{R}^n_x\times\mathbb{R}^n_v)}.$$
		It not only needs the boundedness of the $L^1$ norms of the commuted fields $Z^{\alpha} f$, but also requires an additional integrability in some weighed $L^p$ norms of the commuted fields $Z^{\alpha} f$. Similar norms were used for the low dimensional case, using modified vector fields. 
		
			In our work, we do not need the extra $v$-weighted $L^p$-norms. This is due to our improved estimates, see in particular \lmref{lm.vp.phiZ} and \lmref{lm.mv.grad.term}. Our assumptions are optimal in $v$, in the sense that $f$ integrable in $v$ is required to make sense of the RHS of \eqref{eq.poisson} and \eqref{eq.screened.poisson} classically. 
\end{remark}
\begin{remark}
	From the energy bounds \eqref{hd} and \eqref{lm.d} and the Klainerman-Sobolev inequalities \eqref{ks:id}, \eqref{ks:ab} and \eqref{ks:mvf}, we automatically obtain sharp decay estimates.  For instance, one has for $|\alpha| \le N-n$, 
	$$
		|\rho(\partial_{x}^{\alpha}f)|(t,x)\lesssim \frac{\epsilon}{(1+t+|x|)^{n+|\alpha|}}.
	$$
\end{remark}
\subsection{Previous Work}
The Vlasov-Poisson system is a classical system from plasma physics and we refer to \cite{Glassey.book} for an introduction to its analysis and physical background. In particular, global existence in $3d$ is known for large data thanks to the work of Pfaffelmoser \cite{Pfaffelmoser} and Lions-Perthame \cite{Lions.Perthame}. For small data, the first result was obtained by Bardos-Degond in \cite{Bardos.Degond}. This result was then strengthened first in \cite{Hwang.Rendall.Velazquez} and then \cite{Smulevici.V-P} where sharp decay estimates for derivatives was obtained. This does not follow trivially from \cite{Bardos.Degond}, since commuting the Vlasov equation \eqref{eq.vlasov} with any derivative a priori generates error terms containing $v$ derivatives of $f$ and those typically grow in $t$. An analysis similar to \cite{Hwang.Rendall.Velazquez} was carried out for the Vlasov-Yukawa system by Choi-Ha-Lee \cite{Choi.Ha.Lee.Yukawa}. 

The analysis by Smulevici in \cite{Smulevici.V-P} relied on a generalization of the vector field method of Klainerman \cite{Klainerman} to the case of kinetic transport equations. There has been recently a lot of activities concerning the study of small data Vlasov systems using such methods, see for instance \cite{Fajman.Joudioux.Smulevici.vector-field,Fajman.Joudioux.Smulevici.vlasov-nordstrom,Fajman.Joudioux.Smulevici.Einstein-Vlasov,Bigorgne.maxwell.high,Bigorgne.maxwell.maxwell.3d,Bigorgne.maxwell.massless,Bigorgne.maxwell.asy,Lindblad.Taylor}. See also \cite{Wang} where Fourier techniques and vector field type techniques are used. 

\subsection{Outline of the Paper}

The outline of the paper is as follows: In section 2, we will briefly introduce the commuting vector fields that we use and derive the main properties needed later in the paper; In section 3, we will present our results and the proofs for the V-P system in dimension $n\geq4$ and the V-Y system in dimension $n\geq3$ using the vector field method together with our improvement using the Green's function, see in particular \lmref{lm.vp.phiZ} and \lmref{lm.mv.grad.term}; In section 4, the results and proofs in dimension $n=3$ for V-P and $n=2$ for V-Y are presented using the modified vector fields and our approach to the higher order estimates.  



\section{Preliminaries}


\subsection{Notations}\label{subsec.notation}

Throughout the paper, we denote by $T$ the free transport operator i.e.
$$T(f):=\partial_t f +\sum_{i=1}^{n}v^i\partial_{x^i}f.$$
For any sufficiently regular function $\phi$, let $T_{\phi}$ denote the perturbed transport operator i.e.
$$T_{\phi}(f):=T(f) + \mu\sum_{i=1}^{n}\partial_{x^i}\phi\partial_{v^i}f.$$

We use the Lie bracket $[A,B]:=AB-BA$ to denote the commutation of two operators.

We use the notation $A\lesssim B$ to express that there exist a global constant $C$ such that $A\leq CB$. Here, the global constant $C$ will only depends on the dimension $n$ and the maximum number of commutations.
 
We use the notation $\displaystyle{A\stackrel{*}{\approx}\sum_{i=1}^{d} B_i}$ to express that there exist some global bounded constants $C_i$, such that $\displaystyle{A=\sum_{i=1}^{d}C_i B_i}$, where constant $C_i$ will only depends on the dimension $n$ and the maximum number of commutations.

\subsection{Commutation Vector Fields}

For linear wave equations $\Box\psi=0$, where $\displaystyle{\Box=-\partial_t^2+\sum_{i=1}^n\partial_{x^i}^2}$, the well-know commutation vector fields are consisted of
\begin{enumerate}
	\item Translations in space and time $\partial_t$, $\partial_{x^i}$,
	\item  Rotations $\Omega_{ij}^{x}:=x^i\partial_{x^j}-x^j\partial_{x^i}$,
	\item  Hyperbolic rotations $\Omega_{0i}^{x}:=t\partial_{x^i} + x^i\partial_t$,
	\item  Scaling vector field $\displaystyle{t\partial_t +\sum_{i=1}^n x^i\partial_{x^i}}$.
\end{enumerate}
For the free transport operator $T$, there also exist vector fields that commute with the operator $T$. The simplest examples are the translations $\partial_t$, $\partial_{x^i}$.  In this paper, we will consider the following vector fields (see \cite{Smulevici.V-P})
\begin{enumerate}
	\item Translations in space $\partial_{x^i}$,
	\item  Uniform motions in one direction in microscopic form $t\partial_{x^i} +\partial_{v^i}$,
	\item  Rotations in microscopic form $x^i\partial_{x^j}-x^j\partial_{x^i} + v^i\partial_{v^j} -v^j\partial_{v^i}$,
	\item  Scaling in space in microscopic form $\displaystyle{\sum_{i=1}^n x^i\partial_{x^i} + \sum_{i=1}^nv^i\partial_{v^i}}$.
\end{enumerate}

\begin{remark}
	We do not use all the commutation vector fields in our theorem, only the vector fields that do not contain time derivatives are taken into account. In fact, the scaling in space and time $\displaystyle{t\partial_t +\sum_{i=1}^n x^i\partial_{x^i}}$ also commutes with $T$ in the sense that 
	$$\left[T, \displaystyle{t\partial_t +\sum_{i=1}^n x^i\partial_{x^i}}\right]=T.$$ 
	In \cite{Smulevici.V-P}, a larger family of vector fields including $\partial_t$ and $\displaystyle{t\partial_t +\sum_{i=1}^n x^i\partial_{x^i}}$ are discussed to obtain additional decay on time derivatives of $\nabla\phi$. Similar result can also be obtain with the method used in this paper. To make the paper more concise and simple, we will only study the decay of non-time derivatives.
\end{remark}

\begin{definition}\label{def.vector.fields}
	We denote by $\gamma$ the set of all the above $2n + n(n-1)/2 +1$ vector fields i.e.
	$$\gamma:=\Big\{ \partial_{x^i}, t\partial_{x^i} +\partial_{v^i}, \Omega_{ij}^x + \Omega_{ij}^v, S^x+ S^v,1\leq i<j\leq n\Big\},$$
	where $S^x=\displaystyle{\sum_{i=1}^nx^i\partial_{x^i}}$, $S^v=\displaystyle{\sum_{i=1}^nv^i\partial_{v^i}}$, $\Omega_{ij}^x=x^i\partial_{x^j}-x^j\partial_{x^i}$, $\Omega_{ij}^v=v^i\partial_{v^j}-v^j\partial_{v^i}$.
\end{definition}

\begin{remark}
	Throughout the paper, we make the convention that $\Omega_{ii}^{x}=0$ and $\Omega_{ij}^{x}=-\Omega_{ji}^{x}$ when $i>j$. The same is also for $v-$derivatives.
\end{remark}

For the sake of simplicity, let $Z^i$, $i=1,..., 2n+ n(n-1)/2 +1$, be an ordering of $\gamma$ such that $Z^i=t\partial_{x^i} +\partial_{v^i}$, $i=1,...,n$. We denote by $Z$ a generic commuting vector field in $\gamma$. For any multi-index $\alpha=(\alpha^1,...,\alpha^k)$ with $k=|\alpha|$, the operator $Z^\alpha\in\gamma^{|\alpha|}$ is defined by $Z^{\alpha}=Z^{\alpha^1}Z^{\alpha^2}...Z^{\alpha^k}$. 

\begin{remark}
	In \cite{Smulevici.V-P}, there are also macroscopic vector fields (the set is denoted by $\Gamma$) corresponding to each element in $\gamma$ i.e.
	$$\Gamma:=\Big\{ \partial_{x^i}, t\partial_{x^i}, \Omega_{ij}^x, S^x,1\leq i<j\leq n\Big\}.$$
	For functions that only depend on $(t,x)$, such as $\phi$ and $\rho(f)$, $Z(\phi)$  can be understood as the action of the corresponding macroscopic vector field on $\phi$.
\end{remark}

Before we present our result, let us first look at some important properties of the vector fields $\gamma$.

\begin{lemma}[Commutation with $T$, $\triangle$ and $\rho$]\label{lm.commu.TZ}
	For any $Z^{\alpha}\in\gamma^{|\alpha|}$, we have
	\begin{enumerate}
		\item $[T, Z^{\alpha}]=0$.
		\item $\displaystyle{[Z^{\alpha},\triangle]=\sum_{|\beta|\leq |\alpha|  -1} c_{\alpha\beta} Z^{\beta}\triangle}$.
		\item $\displaystyle{Z^{\alpha}\rho(f)=\rho(Z^{\alpha}f) +\sum_{|\beta|\leq |\alpha|-1}c_{\alpha\beta}' \rho(Z^{\beta} f)}$.
	\end{enumerate}
where $c_{\alpha\beta},c'_{\alpha\beta}$ are global bounded constants that only depend on $n$ and $|\alpha|$.
\end{lemma}

\begin{proof}
The first one is simple, because $\gamma$ is composed of commutation vector field, i.e. $\forall Z\in\gamma$, $[T, Z]=0$.
For the second one, we can check that $[Z,\triangle]=0$ if $Z\neq S^x+S^v$. For $S^x+S^v$, we have $[S^x,\triangle]=-2\triangle$. Similarly, for the last one, we can check that $Z\rho(f)=\rho(Zf)$ for any $Z\neq S^x+S^v$. For $S^x+S^v$, we have
$$S^x\rho(f)=\rho((S^x+S^v)f) + n \rho(f).$$
By induction, we can show the general formula for multi-index operators.
\end{proof}

\begin{lemma}[Commutation within $\gamma$]\label{lm.commu.ZZ}
	For any $Z^{\alpha}\in\gamma^{|\alpha|}$, $Z^{\alpha'}\in\gamma^{|\alpha'|}$, we have,
	$$[Z^{\alpha},Z^{\alpha'}]=\sum_{|\beta|\leq |\alpha| + |\alpha'| -1} c_{\beta}^{\alpha\alpha'}Z^{\beta}.$$
	Moreover, if $Z^{\alpha'}=\partial_{x^i}$, we have,
	$$[Z^{\alpha},\partial_{x^i}]=\sum_{j=1}^{n}\sum_{|\beta|\leq|\alpha|-1}c^{\alpha,i}_{\beta,j}\partial_{x^j}Z^{\beta}.$$
	Here all the constants are global bounded constants that only depend on $n$ and $\max\{|\alpha|,|\alpha'|\}$.
\end{lemma}

\begin{proof}
	For any $Z, Z'\in\gamma$, we can check that $\displaystyle{[Z,Z']\stackrel{*}{\approx} \sum_{Z''\in\gamma}Z''}$, especially, when $Z'=\partial_{x^i}$,  we have $\displaystyle{[Z,\partial_{x^i}]\stackrel{*}{\approx} \sum_{j=1}^n\partial_{x^j}}$. (Here the symbol $\stackrel{*}{\approx}$ means that both sides are equal up to multiplying bounded constants before every terms in the summation, c.f. subsection \ref{subsec.notation}.) We then can obtain the formula by induction on $\alpha,\alpha'$.
\end{proof}

\begin{lemma}[Commutation with $T_{\phi}$]\label{lm.commu.EqT}
	For any $Z^{\alpha}\in\gamma^{|\alpha|}$, we have,
	\begin{equation}
	[T_{\phi}, Z^{\alpha}] f = \sum_{i,j=1}^n\sum_{|\beta|\leq |\alpha|-1}\sum_{|\gamma|+|\beta|\leq |\alpha| }C_{\beta\gamma}^{\alpha,ij} \partial_{x^i} Z^{\gamma} \phi \partial_{v^j} Z^{\beta} f,
	\end{equation}
	where $C_{\beta\gamma}^{\alpha,ij}$ are constants bounded by $n$ and $|\alpha|$.
\end{lemma}

\begin{proof}
	The proof is based on the previous two lemmas. Since $Z^\alpha$ commute with $T$, we only need to look at the term
	$$[\partial_{x^i}\phi\partial_{v^i} , Z^{\alpha}]f=\partial_{x^i}\phi\partial_{v^i}Z^{\alpha}f- Z^{\alpha}(\partial_{x^i}\phi\partial_{v^i}f).$$
	Since $\displaystyle{Z^{\alpha}(\partial_{x^i}\phi\partial_{v^i}f)=\partial_{x^i}\phi Z^{\alpha}(\partial_{v^i}f)  + \sum_{|\gamma|\geq 1,|\beta|+|\gamma|=|\alpha|}C_{\beta\gamma}^{\alpha}Z^{\gamma}(\partial_{x^i}\phi)Z^{\beta}(\partial_{v^i}f)}$, we have,
	$$[\partial_{x^i}\phi\partial_{v^i}, Z^{\alpha}]f=\partial_{x^i}\phi[\partial_{v^i},Z^{\alpha}]f - \sum_{|\gamma|\geq 1,|\beta|+|\gamma|=|\alpha|}C_{\beta\gamma}^{\alpha}Z^{\gamma}(\partial_{x^i}\phi)Z^{\beta}(\partial_{v^i}f).$$
	The only thing left is to show that, for any $Z^{\eta}\in\gamma^{|\eta|}$, $$[\partial_{v^i},Z^{\eta}]\stackrel{*}{\approx}\sum_{|\eta'|\leq|\eta|-1}\partial_{v^j}Z^{\eta'}.$$
	The proof of this is exactly the same as $\partial_{x^i}$ in \lmref{lm.commu.ZZ}.
\end{proof}

\subsection{Klainerman-Sobolev Inequalities for Velocity Averages}

Typically, there are two ways of getting decay estimates for the V-P system or the V-Y system. A standard way is by using the method of characteristics, for example the decay estimate of Bardos-Degond on the V-P system \cite{Bardos.Degond}. Another way is to get decay estimate from the Klainerman-Sobolev inequalities by the using the commutation vector fields, for example Smulevici's result on V-P  system \cite{Smulevici.V-P}. In this paper, we will use the following $L^1$ Klainerman-Sobolev inequality,

\begin{lemma}[$L^1$ Klainerman-Sobolev inequality]\label{lm.klainerman}
	For any sufficiently regular function $\psi(x)$, we have
	\begin{equation}\label{eq.klainerman}
	|\psi|(x)\lesssim\frac{1}{(1+t+|x|)^n}\sum_{|\alpha|\leq n, Z^{\alpha}\in\Gamma^{|\alpha|}}||Z^{\alpha}(\psi)||_{L^1(\mathbb{R}^n_x)}.
	\end{equation}
\end{lemma}

The proof of this inequality is quite standard, we refer to \cite{Sogge.nonlinear.wave} and \cite{Smulevici.V-P} for detailed proof. Here we mention some important equalities about the vector fields.

\begin{lemma}
	For any $1\leq i\leq n$ and $x\neq0$, we have
	\begin{equation}\label{eq.xdx}
	|x|\partial_{x^i}=\sum_{j=1}^n\frac{x^j}{|x|}\Omega_{ji}^{x} +\frac{x^i}{|x|}S^x.
	\end{equation}
\end{lemma}

\begin{lemma}\label{lm.xdx.n}
	For any multi-index $\alpha$, we have
	\begin{equation}
	(t+|x|)^{|\alpha|}\partial_x^{\alpha}=\sum_{|\beta|\leq |\alpha|, Z^{\beta}\in\Gamma^{|\beta|}}C_\beta^\alpha(x) Z^{\beta},
	\end{equation}
	where the coefficients $C_{\beta}(x)$ are homogeneous of degree 0 and uniformly bounded by a constant that depends only on $n$ and $|\alpha|$.
\end{lemma}

\begin{proof}
    Without loss of generality, we assume $x\neq0$. Since $t\partial_x\in\Gamma$, we only need to proof
    $$|x|^{|\alpha'|}\partial_x^{\alpha'}=\sum_{|\beta'|\leq |\alpha'|, Z^{\beta'}\in\Gamma^{|\beta'|}}\widetilde{C}_{\beta'}^{\alpha'}(x) Z^{\beta'}.$$
    Using the equality \eqref{eq.xdx} and \lmref{lm.commu.ZZ}, we can get the desired result by induction on $\alpha'$.
\end{proof}

\begin{proof}[Sketch of the proof of \lmref{lm.klainerman}]
	For fixed $(t,x)$, let $\widetilde{\psi}(y)$ be the function such that
	$$\widetilde{\psi}(y)=\psi(x+(t+|x|)y),\quad y\in B_n(0,1/2).$$
	By Sobolev inequality, we have
	$$|\psi(x)|=|\widetilde{\psi}(0)|\lesssim \sum_{|\alpha|\leq n}\|\partial_y^{\alpha}\widetilde{\psi}\|_{L^1(B_n(0,1/2))}.$$
	Since $|y|\leq 1/2$, we have
	$$\frac{1}{2}(t+|x|)\leq t+|x+(t+|x|)y|\leq \frac{3}{2}(t+|x|).$$
	With the above control and \lmref{lm.xdx.n}, we have that, for every $y\in B_n(0,1/2)$,
	\begin{align*}
	|\partial_y^{\alpha}\widetilde{\psi}(y)| &=(t+|x|)^{|\alpha|}|\partial_x^{\alpha}\psi(x+(t+|x|)y)|\\
	& \lesssim (t+|x+(t+|x|)y|)^{|\alpha|}|\partial_x^{\alpha}\psi(x+(t+|x|)y)|\\
	&\lesssim \sum_{|\beta|\leq|\alpha|,Z^{\beta}\in\Gamma^{|\beta|}}|Z^{\beta}\psi(x+(t+|x|)y)|
	\end{align*}
	By doing the change of variables $z=x+(t+|x|)y$, we will get the decay power in \eqref{eq.klainerman}.
\end{proof}

Now, if we apply the $L^1$ Klainerman-Sobolev inequality to the velocity average $\rho(f)$, together with \lmref{lm.commu.TZ}, we get,

\begin{proposition}
	For any sufficiently regular function $f(x,v)$, we have,
	\begin{equation}
	|\rho(f)|(x)\lesssim \frac{1}{(1+t+|x|)^n}\sum_{|\alpha|\leq n, Z^{\alpha}\in\gamma^{|\alpha|}}\|Z^{\alpha}f\|_{L^1(\mathbb{R}^n_x\times\mathbb{R}^n_v)}.
	\end{equation}
\end{proposition}
 
 Moreover, with the help of  \lmref{lm.xdx.n}, we can get better decay for the derivatives
 
 \begin{proposition}
 	For any sufficiently regular function $f(x,v)$ and multi-index $\alpha$, we have,
 	\begin{equation} \label{ks:id}
 	|\rho(\partial_x^{\alpha}f)|(x)\lesssim \frac{1}{(1+t+|x|)^{n+|\alpha|}}\sum_{|\beta|\leq n +|\alpha|, Z^{\beta}\in\gamma^{|\beta|}}\|Z^{\beta}f\|_{L^1(\mathbb{R}^n_{x}\times\mathbb{R}^n_v)}.
 	\end{equation}
 \end{proposition}

We also recall the following decay estimates from \cite[Proposition 3.3]{Smulevici.V-P}. 
\begin{proposition}
	For any sufficiently regular function $f(x,v)$, we have,
	\begin{equation}\label{ks:ab}
	\rho(|f|)(x)\lesssim \frac{1}{(1+t+|x|)^n}\sum_{|\alpha|\leq n, Z^{\alpha}\in\gamma^{|\alpha|}}\|Z^{\alpha}f\|_{L^1(\mathbb{R}^n_{x}\times\mathbb{R}^n_v)}.
	\end{equation}
\end{proposition}

\subsection{Equations for $Z^{\alpha}\phi$}

For the V-P system, we have

\begin{lemma}\label{lm.poisson.eq}
	Suppose $f$ is sufficiently regular, $\phi$ solves the Poisson equation \eqref{eq.poisson}, then for  any multi-index $\alpha$ and $Z^{\alpha}\in\Gamma^{|\alpha|}$, we have
	\begin{equation}
	\triangle Z^{\alpha}\phi = \sum_{|\beta|\leq|\alpha|}c_{\beta}^{\alpha}\rho(Z^{\beta}f),
	\end{equation}
	where $c_{\beta}^{\alpha}$ are global bounded constants.
\end{lemma}

\begin{proof}
	The proof follows directly from \lmref{lm.commu.TZ} and by the method of induction. We give the proof in the case of the screened Poisson equation just below.
\end{proof}

Similarly, for the V-Y system, we have

\begin{lemma}\label{lm.screen.eq}
	Suppose $f$ is sufficiently regular, $\phi$ solves the screened Poisson equation \eqref{eq.screened.poisson}, then for  any multi-index $\alpha$ and $Z^{\alpha}\in\Gamma^{|\alpha|}$, we have
	\begin{equation}\label{eq.screened.eq}
	\triangle Z^{\alpha}\phi - Z^{\alpha}\phi= \sum_{k=0}^{|\alpha|}\sum_{|\beta|\leq |\alpha|-k} c_{k,\beta}^{\alpha} G_1*^{(k)}\rho (Z^{\beta}f),
	\end{equation}
	where $c_{k,\beta}^{\alpha}$ are global bounded constants, $G_1*^{(k)}$ represent the $k$ times convolution with $G_1$.
\end{lemma}

\begin{proof}
	We prove this by induction. For $|\gamma|=1$, we have,
	\begin{align*}
	\triangle Z \phi - Z \phi & = [\triangle, Z] \phi + Z (\triangle \phi -\phi)\\
	& =c_Z\triangle \phi + Z\rho (f)\\
	& =c_Z(\phi +\rho(f)) + Z\rho (f)\\
	&= (c_Z+1)Z\rho(f)+ c_Z\phi
	\end{align*}
	Here $c_Z$ is a constant such that $[\triangle,Z]=c_Z\triangle$. Since $\phi$ solves \eqref{eq.screened.poisson}, we have
	$$\phi=G_1*\rho(f)$$
	By \lmref{lm.commu.TZ}, we have the desired result for $|\alpha|=1$. Suppose that, for $|\alpha|\leq l$,  we have
	$$	\triangle Z^{\alpha}\phi - Z^{\alpha}\phi= \sum_{k=0}^{|\alpha|}\sum_{|\beta|\leq |\alpha|-k} c_{k,\beta}^{\alpha} G_1*^{(k)}\rho (Z^{\beta}f).$$
	Therefore, we have,
	\begin{align}
	Z^{\alpha}\phi & = G_1*\left( \sum_{k=0}^{|\alpha|}\sum_{|\beta|\leq |\alpha|-k} c_{k,\beta}^{\alpha} G_1*^{(k)}\rho (Z^{\beta}f)\right)\\
	& = \sum_{k=1}^{|\alpha|+1}\sum_{|\beta|\leq |\alpha|-k+1} c_{k,\beta}^{\alpha} G_1*^{(k)}\rho (Z^{\beta}f) \label{eq.screen.solution}
	\end{align}
	Then for $|\alpha|=l+1$, we have
	\begin{align*}
	\triangle Z^{\alpha}\phi -  Z^{\alpha}\phi & =[\triangle,Z^{\alpha}]\phi + Z^{\alpha}(\triangle\phi -\phi)\\
	& = \sum_{|\beta|\leq l}c_{\beta}Z^{\beta}\triangle\phi + Z^{\alpha}\rho(f)\\
	& = \sum_{|\beta|\leq l}c_{\beta}Z^{\beta}(\rho(f) +\phi) + Z^{\alpha}\rho(f)
	\end{align*}
	By \eqref{eq.screen.solution} and \lmref{lm.commu.TZ}, we can get \eqref{eq.screened.eq} for $|\alpha|=l+1$, which ends the proof.
	
\end{proof}


\section{The Higher Dimensional Cases}


In this section, we will proof our main theorem (Theorem \ref{thm.high}) in dimension $n\geq 4$ for the V-P system and $n\geq 3$ for the V-Y system. For any $N$, we consider the following energy
\begin{equation}\label{eq.energy}
\mathcal{E}_N[f]:= \sum_{|\alpha|\leq N, Z^{\alpha}\in\gamma^{|\alpha|}} \|Z^{\alpha}f\|_{L^1(\mathbb{R}^n_x\times\mathbb{R}^n_v)},
\end{equation}
The initial data $f_0$ is assumed to be sufficiently smooth and the energy $\mathcal{E}_N[f_0]\leq \epsilon$ where $\epsilon>0$ is small enough.

\subsection{Bootstrap Assumption}
Let $(f, \phi)$ be the classical solution arising from the initial data $f_0$ and let $T\geq0$ be the largest time such that
\begin{equation} \label{ba:hd}
\mathcal{E}_N[f(t)]\leq 2\epsilon,\quad \forall t\in[0,T].
\end{equation}

\subsection{Improving the Bounds on $\|Z^{\alpha}f\|_{L^1(\mathbb{R}^n_x\times\mathbb{R}^n_v)}$}

To get better bounds on $\|Z^{\alpha}f\|_{L^1(\mathbb{R}^n_x\times\mathbb{R}^n_v)}$, we will use the approximate following conservation law,

\begin{lemma}\label{lm.con.law}
	Suppose $f$ is sufficiently regular function of $(t,x,v)$, then for all $t\in[0,T]$, we have
	\begin{equation}
	\|f(t)\|_{L^1(\mathbb{R}^n_x\times\mathbb{R}^n_v)} \leq \|f(0)\|_{L^1(\mathbb{R}^n_x\times\mathbb{R}^n_v)} +\int_0^t\|T_{\phi}(f)(s)\|_{L^1(\mathbb{R}^n_x\times\mathbb{R}^n_v)}{\rm d}s.
	\end{equation}
\end{lemma}

\begin{proof}
	First, we consider the case where $|f|$ is smooth and has compact support in $(x,v)$ uniformly in $[0,T]$. We have, on the support of $f$, 
	$$T_{\phi}(|f|)=T_{\phi}(f)\frac{f}{|f|}.$$
	We integrate $T_{\phi}(|f|)$ in $(t,x,v)$, 
	\begin{align*}
	\int_0^{t}\int_{x}\int_v T_{\phi}(|f|) & =\int_0^{t}\int_{x}\int_v \Big [\partial_t|f| +\nabla_x\cdot(v|f|) +\mu\nabla_v\cdot (\nabla_x\phi |f|) \Big ]\\
	& =\int_x\int_v |f(t)|-\int_x\int_v|f(0)|
	\end{align*}
	Therefore, we have
	$$	\|f(t)\|_{L^1(\mathbb{R}^n_x\times\mathbb{R}^n_v)} \leq \|f(0)\|_{L^1(\mathbb{R}^n_x\times\mathbb{R}^n_v)} +\int_0^t\|T_{\phi}(f)(s)\|_{L^1(\mathbb{R}^n_x\times\mathbb{R}^n_v)}{\rm d}s.$$
	For more general function $f$, we can approximate $|f|$ by functions $\sqrt{f^2+\eta^2}\chi_\varepsilon(x,v)$, where $\eta > 0$ and $\chi_\varepsilon(x,v)$ is a smooth cut-off function which is $1$ in $B_{x,v}(0,\varepsilon^{-1})$ and $0$ out side of $B_{x,v}(0,2\varepsilon^{-1})$.
	
\end{proof}

From the above lemma, we have that, for $|\alpha|\leq N$,

\begin{equation}
\|Z^{\alpha}f(t)\|_{L^1(\mathbb{R}^n_x\times\mathbb{R}^n_v)} \leq \|Z^{\alpha}f(0)\|_{L^1(\mathbb{R}^n_x\times\mathbb{R}^n_v)} +\int_0^t\|T_{\phi}(Z^{\alpha}f(s))\|_{L^1(\mathbb{R}^n_x\times\mathbb{R}^n_v)} {\rm d}s.
\end{equation}

The initial data is already bounded by $\epsilon$, the main work is to estimate the integral term for $T_{\phi}(Z^{\alpha}f)$. By \lmref{lm.commu.EqT}, we have
$$T_{\phi}(Z^{\alpha}f)=[T_{\phi}, Z^{\alpha}] f = \sum_{i,j=1}^n\sum_{|\beta|\leq |\alpha|-1}\sum_{|\gamma|+|\beta|\leq |\alpha| }C_{\beta\gamma}^{\alpha,ij} \partial_{x^i} Z^{\gamma} \phi \partial_{v^j} Z^{\beta} f.$$
We can always rewrite the $v-$derivatives as linear combination of the commutation vector fields
$$\partial_{v_j}Z^{\beta}f=(t\partial_{x^j} +\partial_{v^j})Z^{\beta}f - t(\partial_{x^j}Z^{\beta}f).$$
Therefore, we have
\begin{equation}\label{eq.high.main}
\|T_{\phi}(Z^{\alpha}f)\|_{L^1(\mathbb{R}^n_x\times\mathbb{R}^n_v)} \lesssim (1+t) \sum_{1\leq|\beta|\leq|\alpha|}\sum_{|\gamma|+|\beta|\leq|\alpha|+1} \|\nabla_x Z^{\gamma}(\phi) Z^{\beta}(f)\|_{L^1(\mathbb{R}^n_x\times\mathbb{R}^n_v)}.
\end{equation}
The only thing we need to estimate now is the term $\|\nabla_x Z^{\gamma}(\phi) Z^{\beta}(f)\|_{L^1(\mathbb{R}^n_x\times\mathbb{R}^n_v)}$.

\subsection{Estimates for $\|\nabla_x Z^{\gamma}(\phi) Z^{\beta}(f)\|_{L^1(\mathbb{R}^n_x\times\mathbb{R}^n_v)}$}

Before we continue the analysis on $\|\nabla_x Z^{\gamma}(\phi) Z^{\beta}(f)\|_{L^1(\mathbb{R}^n_x\times\mathbb{R}^n_v)}$, let us first look at the following lemma

\begin{lemma}\label{lm.decay}
	For any $n\geq 2$, there exist a global constant $C_n>0$ that only depends on $n$, such that for any $x\in\mathbb{R}^n$,
	\begin{equation}
	\int_{\mathbb{R}^n}\frac{1}{|y|^{n-1}(1 + |x+y|)^n}{\rm d}y\leq C_n.
	\end{equation}
\end{lemma}

\begin{proof}
	Without lose of generality, we suppose $x\neq 0$. We divide $\mathbb{R}^n$ into 3 regions: $R_1=\{|y|\leq \frac{2}{3}|x|\}$, $R_2=\{\frac{2}{3}|x|<|y|<2|x|\}$ and $R_3=\{|y|\geq 2|x|\}$. In the region $\{|y|\leq \frac{2}{3}|x|\}$ and $\{|y|\geq 2|x|\}$, we have
	$$|x+y|\geq \big||x|-|y|\big|\geq \frac{|y|}{2},$$
	Therefore, we have 
	$$\int_{R_1\bigcup R_3}\frac{1}{|y|^{n-1}(1 + |x+y|)^n}{\rm d}y\leq \int_{\mathbb{R}^n}\frac{1}{|y|^{n-1}(1 + |y|/2)^n}{\rm d}y = C_0.$$
	In the region $R_2$, first we have a rough estimate,
	\begin{equation*}
	\int_{\frac{2}{3}|x|<|y|<2|x|} \frac{1}{|y|^{n-1}(1 + |x+y|)^n}{\rm d}y\leq \frac{1}{(\frac{2}{3}|x|)^{n-1}}(2|x|)^n\omega_n =2\cdot 3^{n-1} \omega_n |x|,
	\end{equation*}
	which is bounded when $|x|\leq 1$. Here $\omega_n$ denotes the volume of unit ball in $\mathbb{R}^n$. This estimate is bad if $|x|$ is large. Therefore, we need another control of the integral for large $|x|$. Our method is to divide $R_2$ into two subregions: $R_{2,1}=R_2\bigcap\{|y+x|\leq |x|^{1-\delta}\}$ and $R_{2,2}=R_2\bigcap\{|y+x|> |x|^{1-\delta}\}$. Here $\delta>0$ is a fixed number to be chosen.
	In the first subregion, we have
	\begin{equation*}
	\int_{R_{2,1}} \frac{1}{|y|^{n-1}(1 + |x+y|)^n}{\rm d}y\leq \frac{1}{(\frac{2}{3}|x|)^{n-1}}|R_{2,1}| \leq \frac{1}{(\frac{2}{3}|x|)^{n-1}}(|x|^{1-\delta})^n\omega_n =\left(\frac{3}{2}\right)^{n-1} \omega_n |x|^{1-n\delta}.
	\end{equation*}
	In the second subregion, we have
	\begin{equation*}
	\int_{R_{2,2}} \frac{1}{|y|^{n-1}(1 + |x+y|)^n}{\rm d}y\leq \frac{1}{(\frac{2}{3}|x|)^{n-1}|x|^{n(1-\delta)}}(2|x|)^n\omega_n =2\cdot 3^{n-1} \omega_n  |x|^{n\delta-n+1}.
	\end{equation*}
	Let us take $\delta=\frac{1}{2}$, then combining the above two subregions, we have
	$$\int_{R_{2}} \frac{1}{|y|^{n-1}(1 + |x+y|)^n}{\rm d}y\lesssim \omega_n|x|^{1-\frac{n}{2}}.$$
	Since $n\geq 2$, the right hand side is non-increasing when $|x|\rightarrow\infty$. Therefore, the integral is globally bounded by a constant that does not depend on $x$.
	
\end{proof}

With the help of the above lemma, we can estimate the product $\|\nabla_x Z^{\gamma}(\phi) Z^{\beta}(f)\|_{L^1(\mathbb{R}^n_x\times\mathbb{R}^n_v)}$ now.

\begin{lemma}\label{lm.vp.phiZ}
	Suppose $(f,\phi)$ are sufficiently regular solutions to the V-P system. Then, for any multi-index $\gamma,\beta$ with $|\gamma|\leq N$, $|\beta|\leq N$, $|\gamma|+|\beta|\leq N+1$, we have
	\begin{equation}\label{ho}
	\|\nabla_x Z^{\gamma}(\phi) Z^{\beta} (f)\|_{L^1(\mathbb{R}^n_x\times\mathbb{R}^n_v)}\lesssim \sum_{|\beta' |\leq N}\frac{\epsilon}{(1+t)^{n-1}} \|Z^{\beta'}f\|_{L^1(\mathbb{R}^n_x\times\mathbb{R}^n_{v})}
	\end{equation}
\end{lemma}

\begin{proof}
	The idea of the proof is to use the exact formula of $\nabla_xZ^{\gamma}\phi$. For the V-P system, by \lmref{lm.poisson.eq}, we have
	$$\triangle Z^{\gamma}\phi = \sum_{|\gamma'|\leq|\gamma|}c_{\gamma'}^{\gamma}\rho(Z^{\gamma'}f).$$
	Therefore, we have
	$$Z^{\gamma}\phi(x) = \sum_{|\gamma'|\leq |\gamma|}\int_{\mathbb{R}^n}\frac{c_nc_{\gamma'}^{\gamma}}{|y|^{n-2}}\rho(Z^{\gamma'}f)(x-y){\rm d}y.$$
	Then we have
	$$|\nabla_xZ^{\gamma}\phi(x) | \lesssim \sum_{|\gamma'|\leq |\gamma|}\int_{\mathbb{R}^n}\frac{1}{|y|^{n-1}}\rho(|Z^{\gamma'}f|)(x-y){\rm d}y.$$
	As a consequence, we have
	\begin{align*}
	\|\nabla_x Z^{\gamma}(\phi) Z^{\beta} (f)\|_{L^1(\mathbb{R}^n_x\times\mathbb{R}^n_v)} & =\int_{\mathbb{R}^n}|\nabla_xZ^{\gamma}\phi(x) | \rho(|Z^{\beta}f|(x)){\rm d}x\\& \lesssim \sum_{|\gamma'|\leq |\gamma|}\int_{\mathbb{R}^n}\int_{\mathbb{R}^n}\frac{1}{|y|^{n-1}}\rho(|Z^{\gamma'}f|)(x-y)\rho(|Z^{\beta}f|(x)){\rm d}x{\rm d}y.
	\end{align*}
	
	If $|\gamma|\leq N-n$, then by the Klainerman-Sobolev inequality, we have
	$$\rho(|Z^{\gamma'}f|)(x-y)\lesssim \frac{1}{(1+t+|x-y|)^{n}}\sum_{|\beta''|\leq N}\|Z^{\beta''}f\|_{L^1(\mathbb{R}^n_x\times\mathbb{R}^n_v)}\lesssim \frac{\epsilon}{(1+t+|x-y|)^{n}}.$$
	Therefore, we have
	$$\|\nabla_x Z^{\gamma}(\phi) Z^{\beta} (f)\|_{L^1(\mathbb{R}^n_x\times\mathbb{R}^n_v)}\lesssim \int_{\mathbb{R}^n}\int_{\mathbb{R}^n}\frac{1}{|y|^{n-1}}\frac{\epsilon}{(1+t+|x-y|)^{n}}\rho(|Z^{\beta}f|(x)){\rm d}x{\rm d}y.$$
	Now we do the change of variable $y=(1+t)y'$, then the integral becomes
	$$\int_{\mathbb{R}^n}\frac{1}{|y|^{n-1}}\frac{\epsilon}{(1+t+|x-y|)^{n}}{\rm d}y=\frac{\epsilon}{(1+t)^{n-1}}\int_{\mathbb{R}^n}\frac{1}{|y'|^{n-1}(1+|y'-\frac{x}{1+t}|)^n}{\rm d}y'.$$
	By \lmref{lm.decay}, we have
	$$\int_{\mathbb{R}^n}\frac{1}{|y|^{n-1}}\frac{\epsilon}{(1+t+|x-y|)^{n}}{\rm d}y\lesssim \frac{\epsilon}{(1+t)^{n-1}}$$
	Therefore, 
	$$\|\nabla_x Z^{\gamma}(\phi) Z^{\beta} (f)\|_{L^1(\mathbb{R}^n_x\times\mathbb{R}^n_v)}\lesssim \sum_{|\beta|\leq N}\frac{\epsilon}{(1+t)^{n-1}} \|Z^{\beta}f\|_{L^1(\mathbb{R}^n_x\times\mathbb{R}^n_v)}$$
	Now, if $|\gamma|>N-n$, since $|\beta|+|\gamma|\leq N+1$ and $N\geq 2n$, we have $|\beta|\leq N-n$. We do change of variable $z=x-y$ and do the same analysis to $\rho(|Z^\beta f|)$, we can get the same decay estimate.\\
	
\end{proof}

For the V-Y system, we can get even better estimates. Let us remind that a solution to \eqref{eq.screen.solution} can be written as
$$\phi(x)=\int_{\mathbb{R}^n}G_1(x-y)\rho(f)(y){\rm d}y,$$
where $G_1(x)$ has the form
$$G_1(x)=\frac{c'_n}{|x|^{\frac{n}{2}-1}}K_{\frac{n}{2}-1}(|x|),$$
Here $c'_n$ is a constant and $K_{\nu}(r)$ is the modified Bessel function of the second kind, which has an integral expression of the form when $r>0$
$$K_{\nu}(r)=\int_0^{\infty}e^{-r\cosh \lambda} \cosh (\nu \lambda){\rm d}\lambda.$$ 
The asymptotics of $K_\nu(r)$  are well known(see \cite{Bessel}). Let us first prove the following simple estimate on $K_\nu(r)$.

\begin{lemma}
	For any $\nu\geq \frac{1}{2}$ and $r>0$, we have
	\begin{equation}\label{eq.Ku}
	K_{\nu}(r)\lesssim \frac{e^{-r}}{\sqrt{r}}\left(1+\frac{1}{r^{\nu-\frac{1}{2}}}\right).
	\end{equation}
\end{lemma}

\begin{proof}
	This can be shown through change of variables in the integral form of $K_\nu(r)$. We recall that
	$$K_{\nu}(r)=\int_0^{\infty}e^{-r\cosh \lambda} \cosh (\nu \lambda){\rm d}\lambda.$$ 
	Let us do the change of variable $u=2\sqrt{r}\sinh (\frac{\lambda}{2})$, then we have
	$${\rm d}u = \sqrt{r}\cosh (\frac{\lambda}{2}){\rm d}\lambda,\quad \cosh (\frac{\lambda}{2})=\left(1+\frac{u^2}{4r}\right)^{\frac{1}{2}},$$
	$$\cosh \lambda=1+2\sinh^2 (\frac{\lambda}{2})=1+ \frac{u^2}{2r}.$$
	Because $e^{\nu \lambda}=(e^{\frac{\lambda}{2}})^{2\nu}\leq (2\cosh (\frac{\lambda}{2}))^{2\nu}$, then we have
	$$\cosh (\nu \lambda)\leq e^{\nu \lambda}\lesssim (\cosh (\frac{\lambda}{2}))^{2\nu}.$$
	Therefore, we have
	$$K_{\nu}(r)\lesssim \frac{e^{-r}}{\sqrt{r}}\int_0^{\infty}e^{-\frac{u^2}{2}}\left(1+\frac{u^2}{4r}\right)^{\nu -\frac{1}{2}}{\rm d}u.$$
	Since $\nu\geq 1/2$, we have
	$$\left(1+\frac{u^2}{4r}\right)^{\nu -\frac{1}{2}}\lesssim 1+ \frac{u^{2\nu -1}}{r^{\nu-\frac{1}{2}}}.$$
	Therefore, we have
	$$K_{\nu}(r)\lesssim \frac{e^{-r}}{\sqrt{r}}\left(1+\frac{1}{r^{\nu-\frac{1}{2}}}\right).$$
\end{proof}

\begin{lemma}\label{lm.V-Y.L1}
	For any $1\leq p\leq \infty$ and any function $\Psi\in L^p(\mathbb{R}^n_x)$, we have
	$$\|G_1*\Psi\|_{L^p}\lesssim \|\Psi\|_{L^p},\quad \|\nabla_xG_1*\Psi\|_{L^p}\lesssim \|\Psi\|_{L^p}.$$
\end{lemma}

\begin{proof}
	This is a direct result from Young's convolution inequality. The only thing we need to show is that $\|G_1\|_{L^1}$ and $\|\nabla_xG_1\|_{L^1}$ are finite. To show this, it suffices to use the estimate \eqref{eq.Ku} and the following relation
	$$\frac{\partial}{\partial r}K_{\nu}(r)=-\frac{1}{2}(K_{\nu-1}(r)+K_{\nu+1}(r)).$$
\end{proof}

We then have the following lemma, similar to \lmref{lm.vp.phiZ}.
\begin{lemma}\label{lm.vy.phiZ}
	Suppose $(f,\phi)$ are sufficiently regular solutions to the V-Y system. Then, for any multi-index $\gamma,\beta$ with $|\gamma|\leq N$, $|\beta|\leq N$, $|\gamma|+|\beta|\leq N+1$, we have
	\begin{equation}\label{eq.V-Y.high}
	\|\nabla_x Z^{\gamma}(\phi) Z^{\beta} (f)\|_{L^1(\mathbb{R}^n_x\times\mathbb{R}^n_v)}\lesssim \sum_{|\beta' |\leq N}\frac{\epsilon}{(1+t)^{n}} \|Z^{\beta'}f\|_{L^1(\mathbb{R}^n_x\times\mathbb{R}^n_{v})}
	\end{equation}
\end{lemma}

\begin{proof}
	For the V-Y system, by \lmref{lm.screen.eq}, we have
	$$\triangle Z^{\gamma}\phi - Z^{\gamma}\phi= \sum_{k=0}^{|\gamma|}\sum_{|\gamma'|\leq |\gamma|-k} c_{k,\gamma'}^{\gamma} G_1*^{(k)}\rho (Z^{\gamma'}f),$$
	where $G_1(x)=\frac{c'_n}{|x|^{\frac{n}{2}-1}}K_{\frac{n}{2}-1}(|x|)$, and $K_{\nu}(r)$ is the modified Bessel function of the second kind.
	Therefore, we have
	$$\nabla_xZ^{\gamma}\phi = \nabla_xG_1*\left(\sum_{k=0}^{|\gamma|}\sum_{|\gamma'|\leq |\gamma|-k} c_{k,\gamma'}^{\gamma} G_1*^{(k)}\rho (Z^{\gamma'}f)\right).$$
	Now, if $|\gamma|\leq N-n$, by \lmref{lm.V-Y.L1} and Klainerman-Sobolev inequality, we have
	$$\|\nabla_x Z^{\gamma}\phi\|_{L^\infty}\lesssim \sum_{|\gamma'|\leq N-n}\|\rho(Z^{\gamma'}f)\|_{L^\infty}\lesssim \frac{\epsilon}{(1+t)^n},$$
	which gives \eqref{eq.V-Y.high}. If $|\gamma|>N-n$, since $|\beta| +|\gamma\leq N+1$ and $N\geq 2n$, we have $|\beta|\leq N-n$. Therefore, we have
	$$\|\nabla_x Z^{\gamma}\phi\|_{L^1}\lesssim \sum_{|\gamma'|\leq N}\|\rho(Z^{\gamma'}f)\|_{L^1},\quad \|\rho(Z^{\beta}f)\|_{L^\infty}\lesssim \frac{\epsilon}{(1+t)^n},$$
	which also gives \eqref{eq.V-Y.high}.
\end{proof}

\subsection{Improving the Bootstrap Assumption}

Now, with the estimates for $\|\nabla_x Z^{\gamma}(\phi) Z^{\beta}(f)\|_{L^1(\mathbb{R}^n_x\times\mathbb{R}^n_v)}$, we can finally improve the bound on $\mathcal{E}_N[f]$. For the V-P system, we have
\begin{align*}
\|T_{\phi}(Z^{\alpha}f)(s)\|_{L^1(\mathbb{R}^n_x\times\mathbb{R}^n_v)} & \lesssim (1+s) \sum_{1\leq|\beta|\leq|\alpha|}\sum_{|\gamma|+|\beta|\leq|\alpha|+1} \|\nabla_x Z^{\gamma}(\phi) Z^{\beta}(f)\|_{L^1(\mathbb{R}^n_x\times\mathbb{R}^n_v)}\\
& \lesssim\frac{\epsilon}{(1+s)^{n-2}}\sum_{|\beta'|\leq N}\|Z^{\beta'}f\|_{L^1(\mathbb{R}^n_x\times\mathbb{R}^n_v)}\\
&\lesssim \frac{\epsilon^2}{(1+s)^{n-2}}.
\end{align*}
Therefore, we have
\begin{align*}
\mathcal{E}_N[f(t)]& \leq \mathcal{E}_N[f_0] + C\int_0^t \frac{\epsilon^2}{(1+s)^{n-2}}{\rm d}s,
\end{align*}
where the constant $C$ only depends on $n$ and $N$. Since $n\geq 4$, the integral is globally bounded for any $T$. For $\epsilon$ small enough such that
$$C\int_0^\infty\frac{\epsilon}{(1+s)^{n-2}}{\rm d}s\leq \frac{1}{2},$$
we  have the improved bound
$$\mathcal{E}_N[f(t)]\leq \epsilon + \frac{\epsilon}{2}<2\epsilon,\quad\forall t\in[0,T],$$
which indicates that $T$ is not finite. Similarly, for V-Y system, we have better decay
$$\|T_{\phi}(Z^{\alpha}f)(s)\|_{L^1(\mathbb{R}^n_x\times\mathbb{R}^n_v)}\lesssim \frac{\epsilon^2}{(1+s)^{n-1}},$$
therefore we can get the same result in dimensions $n\geq 3$.

This improves the bootstrap assumption \eqref{ba:hd} and therefore ends the proof of Theorem \ref{thm.high}.


\section{3-Dimensional V-P System and 2-Dimensional V-Y System}


In this section, we will deal with the low dimensional cases, i.e., the V-P system in 3-dimensions and the V-Y system in 2-dimensions, cf~Theorem \ref{thm.low}. The standard vector field method will not work in this low dimensional cases. The main difficulty is that, for example, the V-P system in 3-dimensions, using the vector field method in previous section only provides us the estimate:
$$\mathcal{E}_N[f(t)] \leq \mathcal{E}_N[f_0] + C\int_0^t \frac{\epsilon^2}{1+s}{\rm d}s,$$
where the integral on the right hand side is not uniformly bounded in $t$. To get rid of this problem and slightly improve the estimate, we need to do some modification to our vector fields so that we can get better decay than $\frac{1}{1+s}$ for the error term (see \cite{Smulevici.V-P}). 

In this section, if $n=3$, we work on the V-P system, and if $n=2$ we work on the V-Y system. $f_0$ denotes initial data satisfying the assumptions of Theorem \ref{thm.low} and $(f, \phi)$ denotes the classical solution arising from it.

\subsection{The Modified Vector Fields}

Let us recall that our vector field is composed of
$$\gamma=\Big\{ \partial_{x^i}, t\partial_{x^i} +\partial_{v^i}, \Omega_{ij}^x + \Omega_{ij}^v, S^x+ S^v,1\leq i<j\leq n\Big\},$$
We let $Z^i$, $i=1,..., 2n+ n(n-1)/2 +1$, be an ordering of $\gamma$ such that $Z^i=t\partial_{x^i} +\partial_{v^i}$, $i=1,...,n$.   Let us first look at the commutation with operator $T_\phi$. For each vector field $Z^i$, we can easily calculate that
$$[T_\phi, Z^i] (f)= -\mu\sum_{k=1}^n\partial_{x^k}(Z^i\phi +c_i\phi)\partial_{v^k}f. $$
Here $c_i=-2$ if $Z^i=S^x+S^v$, otherwise, we have $c_i=0$. If we want to write the right hand side as linear combinations of vector fields on $f$, we have to rewrite $\partial_{v^k}$ as
$$\partial_{v^k} = (t\partial_{x^k} +\partial_{v^k}) - t(\partial_{x^k}).$$
Therefore, we gain an additional multiplier $(1+t)$ in our decay estimate, that is the reason why our method will not work for $n=3$ (or $n=2$ for the V-Y system). To avoid this problem, we consider a modification of our vector fields in the following form:
$$Y^i:=Z^i-\sum_{k=1}^n\varphi_k^i(t,x,v)\partial_{x^k},\quad i=1,..., 2n+ n(n-1)/2 +1,$$
where $\varphi^i_k$ are sufficiently smooth functions of $(t,x,v)$ and vanish at $t=0$. To find out what the functions $\varphi_k^i$ should be, let us commute $T_\phi$ with the modified vector field $Y^i$. For each $Y^i$, we have that,
\begin{align}\label{eq.modi.TY.1}
[T_\phi, Y^i](f) & =-\mu\sum_{k=1}^n\partial_{x^k}(Z^i\phi +c_i\phi)\partial_{v^k}f  -\mu\sum_{k=1}^n T_\phi(\varphi_k^i)\partial_{x^k}f +\mu \sum_{k,j=1}^n\varphi_k^i\partial_{x^k}\partial_{x^j}\phi\partial_{v^j}f.
\end{align}
Now, if we rewrite $\partial_{v^k}$ as
$$\partial_{v^k} = (t\partial_{x^k} +\partial_{v^k}) - t(\partial_{x^k}).$$
We will get
\begin{align*}
[T_\phi, Y^i](f) & =-\mu\sum_{k=1}^n\partial_{x^k}(Z^i\phi +c_i\phi)Z^k f +\mu t\sum_{k=1}^n\partial_{x^k}(Z^i\phi +c_i\phi)\partial_{x^k} f \\
& \quad -\mu\sum_{k=1}^n T_\phi(\varphi_k^i)\partial_{x^k}f +\mu \sum_{k,j=1}^n\varphi_k^i\partial_{x^k}\partial_{x^j}\phi Z^j f -\mu \sum_{k,j=1}^n\varphi_k^i\partial_{x^k}(t\partial_{x^j})\phi \partial_{x^j} f 
\end{align*}
The only bad term with addition multiplier of $t$ is $\mu t\partial_{x^k}(Z^i\phi +c_i\phi)\partial_{x^k} f$ when $Z^i\neq \partial_x$. To remove these bad terms, we just need $T_\phi(\varphi_k^i)=\mu t\partial_{x^k}(Z^i\phi +c_i\phi)$. 
\begin{definition}\label{def.modified}
	Let $Z^i$, $i=1,..., 2n+ n(n-1)/2 +1$, be an ordering of $\gamma$. The modified vector fields $Y^i$ have the following form,
	$$Y^i:=Z^i-\sum_{k=1}^n\varphi_k^i(t,x,v)\partial_{x^k},\quad i=1,..., 2n+ n(n-1)/2 +1,$$
	where 
	\begin{enumerate}
		\item $\varphi_k^i\equiv 0$ if $Z^i$ is a translation in space, i.e. $Z^i=\partial_{x^k}$.
		\item If $Z^i$ is not a translations, then $\varphi_k^i(t,x,v)$ is a solution of
		\begin{gather}
		T_\phi(\varphi_k^i)=\mu t\partial_{x^k}(Z^i\phi +c_i\phi),\label{eq.modi.phi.eq}\\
		\varphi_k^i(0,x,v)=0.
		\end{gather}
		Here $c_i=-2$ if $Z^i=S^x+S^v$, otherwise, we have $c_i=0$.
	\end{enumerate}
    The set of modified vector fields is denoted by $\gamma_m$.
\end{definition}

Throughout the paper, we denote by $Y$ a generic modified vector field in $\gamma_m$. For any multi-index $\alpha=(\alpha^1,...,\alpha^k)$ with $k=|\alpha|$, the operator $Y^\alpha\in\gamma^{|\alpha|}_m$ is defined by $Y^{\alpha}=Y^{\alpha^1}Y^{\alpha^2}...Y^{\alpha^k}$. We also denote by $\mathcal{M}$ the set of all functions $\{\varphi_k^i\}$. We also denote by $\varphi$ a generic function in $\mathcal{M}$. We say that $P(\varphi)$ is a multi-linear form of degree $d$ with signature less than $k$ if $P(\varphi)$ has the following form
$$P(\varphi)=\sum_{|\alpha_1|+...+|\alpha_d|\leq k\atop (\varphi_1,...,\varphi_d)\in\mathcal{M}^d}C_{\bar{\alpha}\bar{\varphi}}\prod_{j=1}^d Y^{\alpha_j}(\varphi_j)$$
where $\alpha_j$ are all multi-indices and $C_{\bar{\alpha}\bar{\varphi}}$ are constants with $\bar{\alpha}=(\alpha_1,...,\alpha_d)$ and $\bar{\varphi}=(\varphi_1,...,\varphi_d)$.

\subsection{Properties of Modified Vector Fields}

In this section, we will study the main properties of the modified vector field that will be used later in our proof. 

\begin{lemma}[Higher order commutation formula with $T_\phi$]
	For any multi-index $\alpha$, we have
	\begin{equation}\label{eq.mv.tphi}
	[T_{\phi}, Y^{\alpha}] = \sum_{d=0}^{|\alpha|+1}\sum_{i=1}^{n}\sum_{|\gamma|,|\beta|\leq |\alpha|} P_{d\gamma\beta}^{\alpha,i}(\varphi)\partial_{x^i}Z^{\gamma}(\phi) Y^{\beta},
	\end{equation}
	where $P_{d\gamma\beta}^{\alpha,i}(\varphi)$ are multilinear forms of degree $d$ with signatures less than $k$ such that $k\leq |\alpha|-1$ and $k+|\gamma|+|\beta|\leq |\alpha| +1$.
\end{lemma}

\begin{proof}
	Let us first look at the case when $|\alpha|=1$. By \eqref{eq.modi.TY.1} and rewriting $\partial_{v^k}$ as
	$$\partial_{v^k} = (t\partial_{x^k} +\partial_{v^k}) - t(\partial_{x^k}),$$
	we will get, after using the equations on $\varphi_k^i$ \eqref{eq.modi.phi.eq},
	\begin{align*}
	[T_\phi, Y^i](f) 
	&= -\mu\sum_{k=1}^n\partial_{x^k}(Z^i\phi +c_i\phi)Z^k f +\mu \sum_{k,j=1}^n\varphi_k^i\partial_{x^k}\partial_{x^j}\phi Z^j f -\mu \sum_{k,j=1}^n\varphi_k^i\partial_{x^k}(t\partial_{x^j})\phi \partial_{x^j} f
	\end{align*}
	Now, by just rewriting $Z^k$ as
	$$Z^k=Y^k + \sum_{l=1}^n\varphi_l^k\partial_{x^l},$$
	we will get that $[T_\phi,Y^i]$ has the following form
	$$[T_\phi,Y^i]=\sum_{|\gamma|\leq 1\atop 1\leq k,l\leq n}P(\varphi)\partial_{x^k}Z^{\gamma}(\phi)Y^l,$$ where $P(\varphi)$ is a multi-linear form of degree at most 2 with signature 0. So the lemma is true for $|\alpha|=1$. For $|\alpha|\geq 2$, we can prove the lemma by induction with the help of following equality
	$$[T_\phi, Y^iY^\alpha ]=[T_\phi, Y^i]Y^\alpha + Y^i [T_{\phi}, Y^{\alpha}].$$
	
\end{proof}

\begin{lemma}\label{lm.mv.Z.Y}
	For any multi-index $\alpha$, we have
	\begin{equation}
	Z^{\alpha}=\sum_{d=0}^{|\alpha|}\sum_{|\beta|\leq |\alpha|}P_{d\beta}^{\alpha}(\varphi)Y^{\beta},
	\end{equation}
	where $P_{d\beta}^{\alpha}(\varphi)$ are multilinear forms of signature less than $k$ with $k\leq |\alpha|-1$ and $k+|\beta|\leq |\alpha|$.
\end{lemma}

\begin{proof}
	The lemma is trivial for $|\alpha|=1$, since $\displaystyle{Z^i=Y^i +\sum_{k=1}^n\varphi_k^i\partial_{x^k}}$ where $\partial_{x^k}$ is also a modified vector field. The rest can be shown by induction.
\end{proof}

\begin{lemma}\label{lm.mv.rho.Z}
	For all multi-index $\alpha$, we have
	\begin{equation}
	\rho(Z^{\alpha} f) =\sum_{d=0}^{|\alpha|}\sum_{|\beta|\leq |\alpha|}\rho(Q_{d\beta}^{\alpha}(\partial_x \varphi)Y^{\beta} f) +\sum_{j=1}^{|\alpha|}\sum_{d=1}^{|\alpha|+1}\sum_{|\beta|\leq |\alpha|}\frac{1}{t^j}\rho(P_{d\beta}^{\alpha,j}(\varphi) Y^{\beta}(f)),
	\end{equation}
	where $Q_{d\beta}^{\alpha}(\partial_x \varphi)$ are multilinear forms with respect to $\partial_x \varphi$ of signatures less than $k'$ satisfying $k'\leq |\alpha|-1$ and $k' + d +|\beta|\leq |\alpha|$, $P_{d\beta}^{\alpha,j}(\varphi)$ are multilinear froms of degree $d$ with signatures less than $k$ satisfying $k\leq |\alpha| $ and $k+|\beta|\leq |\alpha|$.
\end{lemma}
 
\begin{proof}
	This lemma can be shown by induction. We refer to \cite[Lemma 6.3]{Smulevici.V-P} for the proof.
\end{proof}

\begin{lemma}\label{lm.mv.Y.na}
	We have
	\begin{equation}
	Y^{\alpha}\nabla \phi = Z^{\alpha}\nabla \phi + \frac{1}{t}\sum_{d=1}^{|\alpha|}\sum_{|\beta|\leq |\alpha|}P_{d\beta}^{\alpha}(\varphi)Z^{\beta}\nabla\phi
	\end{equation}
	where $P_{d\beta}^{\alpha}(\varphi)$ are multilinear forms of degree $d$ with signatures less than $k$ satisfying $k\leq |\alpha|-1$ and $k+|\beta|\leq |\alpha|$.
\end{lemma}

\begin{proof}
	For $|\alpha|=1$, we have
	$$Y\nabla \phi= (Z+\varphi\partial_x)\nabla \phi=Z\nabla \phi +\frac{1}{t}\varphi(t\partial_x)\nabla \phi.$$
	The rest can be shown by induction.
\end{proof}

\subsection{Bootstrap Assumptions}

In this section, we will use the energy corresponding to modified vector fields. For $N \ge 2n+3$, we take the following energy
\begin{equation}
\mathcal{E}_N[f]= \sum_{|\alpha|\leq N} \|Y^{\alpha}f\|_{L^1(\mathbb{R}^n_x\times\mathbb{R}^n_v)}.
\end{equation}

	

The modified vector field works for general cases, but here we only consider the low dimensional ones, i.e., if $n=3$, we work on the V-P system, and if $n=2$ we work on the V-Y system.

We consider the following bootstrap assumptions. Let $T\geq 0$ be the largest time so that, for all $t\in[0,T]$, we have
\begin{enumerate}
	\item 
	\begin{equation}\label{eq.bootstrap.1}
	\mathcal{E}_N[f(t)]\leq 2\epsilon.
	\end{equation}
	\item  For any multi-index $\alpha$ with $|\alpha|\leq N-n-2$ and any $Y^\alpha\in\gamma_m^{|\alpha|}$, we have
	\begin{equation}\label{eq.bootstrap.2}
	|Y^{\alpha}\varphi(t,x,v)|\lesssim \epsilon^{\frac{1}{2}} (1+\log(1+t)), \quad \forall (x,v)\in\mathbb{R}^n_x\times\mathbb{R}^n_v.
	\end{equation}
	\item  For any multi-index $\alpha$ with $|\alpha|\leq N-n-3$ and any $Y^\alpha\in\gamma_m^{|\alpha|}$, we have
	\begin{equation}\label{eq.bootstrap.3}
	|Y^{\alpha}\nabla_x\varphi(t,x,v)|\lesssim \epsilon^{\frac{1}{2}}, \quad \forall (x,v)\in\mathbb{R}^n_x\times\mathbb{R}^n_v.
	\end{equation}
	\item  For any multi-index $\alpha$ with $|\alpha|\leq N-n-1$ and any $Z^{\alpha}\in\Gamma^{|\alpha|}$, we have
	\begin{equation}\label{eq.bootstrap.4}
	|\nabla_xZ^{\alpha}\phi(t,x)|\lesssim \frac{\epsilon^{\frac{1}{2}}}{(1+t)^2}, \quad \forall x\in\mathbb{R}^n.
	\end{equation}         
\end{enumerate}

\begin{remark}
	The modified vector field $Y^i=Z^i$ at time $t=0$. Therefore, 
	$$\mathcal{E}_N[f_0]=\sum_{|\alpha|\leq N}\|Z^\alpha f_0\|_{L^1(\mathbb{R}^n_x\times\mathbb{R}^n_v)} \leq \epsilon.$$
\end{remark}

\subsection{Klainerman-Sobolev Inequality for Modified Vector Field}

Based on the bootstrap assumptions on $\varphi$, we can get the Klainerman-Sobolev inequality for the modified vector field. We refer to \cite[Proposition 6.1]{Smulevici.V-P} for the proof. 

\begin{proposition}
	For any sufficiently smooth function $f$, we have
	\begin{equation} \label{ks:mvf}
	\rho(|f|)(t,x)\lesssim\frac{1}{(1+t+|x|)^n} \sum_{|\alpha|\leq n} \|Y^{\alpha}f\|_{L^1(\mathbb{R}^n_x\times\mathbb{R}^n_v)}.
	\end{equation}
	
\end{proposition}

\subsection{Estimates for $\|Y^{\alpha}(\varphi)Y^{\beta}(f)\|_{L^1(\mathbb{R}^n_x\times\mathbb{R}^n_v)}$}

Before we prove the main result, let us first estimate the products of type $Y^{\alpha}(\varphi)Y^{\beta}(f)$, which will play an important role in improving the bootstrap assumptions in next subsection. Let us first prove the following lemma,

\begin{lemma}\label{lm.mv.grad.term}
	For any multi-index $\gamma,\alpha$ with $|\gamma|\leq N$, $|\alpha|\leq N-n$, we have
	\begin{equation}\label{hoe}
	\|\nabla_x Z^{\gamma}(\phi) Y^{\alpha} (f)\|_{L^1_{(\mathbb{R}^n_x\times\mathbb{R}^n_v)}}\lesssim \sum_{|\beta|\leq |\gamma|}\frac{\epsilon}{(1+t)^2} \|\rho(Z^{\beta}f)\|_{L^1(\mathbb{R}^n_x)}
	\end{equation}
	where $(f,\phi)$ is a sufficiently smooth solution to V-P system if $n=3$ and V-Y system if $n=2$.
\end{lemma}

\begin{proof}
	The proof is exactly the same as \lmref{lm.vp.phiZ} and \lmref{lm.vy.phiZ}, where we need to use the exact formula for $\nabla_xZ^{\gamma}\phi$. The only difference is to use the Klainerman-Sobolev inequality for the modified vector field of $Y^\alpha f$. 
\end{proof}

\begin{remark}
	Note that the norm on the RHS of \eqref{hoe} is weaker than that appearing in \eqref{ho}. The reason is that since $|\alpha| \le N-n$, we can always appeal the Klainerman-Sobolev estimate to control the contribution of $|Y^{\alpha} (f)|$.
\end{remark}

With the help of above lemma, we can show that

\begin{lemma}\label{lm.mv.YY}
	For any fixed small number $\sigma>0$, there exist constants $C_{\sigma}$ and $\epsilon_{\sigma}$ such that, if $\epsilon\leq\epsilon_{\sigma}$, then for all multi-index $\alpha,\beta$ with $|\alpha|\leq N-1$, $|\beta| \leq N$ and $|\alpha| + |\beta| \leq N+1$, we have
	\begin{equation}
	\|Y^{\alpha}(\varphi)(t) Y^{\beta}(f)(t)\|_{L^1(\mathbb{R}^n_x\times\mathbb{R}^n_v)}\leq C_{\sigma}(1+t)^{\sigma}\epsilon,
	\end{equation}
	Moreover, for all multi-index $\alpha,\beta$ with $|\alpha|\leq N-2$, $|\beta| \leq N$ and $|\alpha| + |\beta| \leq N$, and all $1\leq i\leq n$, we have
	\begin{equation}
	\|Y^{\alpha}(\partial_{x^i}\varphi)(t) Y^{\beta}(f)(t)\|_{L^1(\mathbb{R}^n_x\times\mathbb{R}^n_v)}\leq C_{\sigma}\epsilon.
	\end{equation}
	
\end{lemma}

\begin{proof}
	For simplicity of writing, let us denote
	$$\mathcal{F}(t)= \sum_{|\alpha|\leq N-1}\sum_{|\beta|\leq N\atop |\alpha|+|\beta|\leq N+1}\|Y^{\alpha}(\varphi)(t) Y^{\beta}(f)(t)\|_{L^1(\mathbb{R}^n_x\times\mathbb{R}^n_v)},$$
	$$\mathcal{G}(t)= \sum_{|\alpha|\leq N-2}\sum_{i=1}^n\sum_{ |\alpha|+|\beta|\leq N}\|Y^{\alpha}(\partial_{x^i}\varphi)(t) Y^{\beta}(f)(t)\|_{L^1(\mathbb{R}^n_x\times\mathbb{R}^n_v)}.$$
	We only need to estimate $\mathcal{F}$ and $\mathcal{G}$. First, we have that, if $|\alpha|\leq N-n-2$, by bootstrap assumptions \eqref{eq.bootstrap.1}, \eqref{eq.bootstrap.2}, we have
	\begin{align*}
	\|Y^{\alpha}(\varphi)(t) Y^{\beta}(f)(t)\|_{L^1(\mathbb{R}^n_x\times\mathbb{R}^n_v)} & \lesssim \epsilon^{\frac{1}{2}} (1+\log(1+t))\|Y^{\beta}(f)\|_{L^1(\mathbb{R}^n_x\times\mathbb{R}^n_v)}\\
	& \lesssim (1+t)^{\sigma_0} \epsilon^{\frac{3}{2}},
	\end{align*}
	where $\sigma_0>0$ is a very small constant that is to be fixed later.
	Similarly, for $|\alpha|\leq N-n-3$, we have
	$$\|Y^{\alpha}(\partial_{x^i}\varphi)(t) Y^{\beta}(f)(t)\|_{L^1(\mathbb{R}^n_x\times\mathbb{R}^n_v)}\lesssim \epsilon^{\frac{3}{2}}.$$
	If $|\alpha|>N-n-2$, then we have $|\beta|\leq N-n-1$ since $N\geq2n +3$. We cannot use the bootstrap assumptions \eqref{eq.bootstrap.2} since $|\alpha|$ is too large. However, we can use the Klainerman-Sobolev inequality since $|\beta|$ is not too large. To estimate $\|Y^{\alpha}(\varphi)(t) Y^{\beta}(f)(t)\|_{L^1(\mathbb{R}^n_x\times\mathbb{R}^n_v)}$ (including the special case $Y^{\alpha}=Y^{\alpha'}\partial_{x^i}$), we will use the approximate conservation law in \lmref{lm.con.law}. Since $\varphi=0$ at time $t=0$, we have
	$$\|Y^{\alpha}(\varphi)(t) Y^{\beta}(f)(t)\|_{L^1(\mathbb{R}^n_x\times\mathbb{R}^n_v)} \leq \int_0^t \|T_{\phi}(Y^{\alpha}(\varphi) Y^{\beta}(f))\|_{L^1(\mathbb{R}^n_x\times\mathbb{R}^n_v)}(s){\rm d}s.$$
	So the main issue is to deal with $T_{\phi}(Y^{\alpha}(\varphi) Y^{\beta}(f))$, which can be divided into 3 parts $I_1$, $I_2$ and $I_3$
	\begin{align*}
	T_{\phi}(Y^{\alpha}(\varphi) Y^{\beta}(f)) & =Y^{\alpha}(\varphi)T_{\phi}(Y^{\beta}(f)) + [T_{\phi}, Y^{\alpha}](\varphi) Y^{\beta}(f) + Y^{\alpha}T_{\phi }(\varphi) Y^{\beta}(f)\\
	& =I_1+I_2 +I_3
	\end{align*}
	
	{\bf Estimate for $\|I_1\|_{L^1(\mathbb{R}^n_x\times\mathbb{R}^n_v)}$.} By \eqref{eq.mv.tphi}, we have
	$$I_1=\sum_{d=0}^{|\beta|+1}\sum_{i=1}^{n}\sum_{|\gamma|,|\beta'|\leq |\beta|} P_{d\gamma\beta'}^{\beta,i}(\varphi)\partial_{x^i}Z^{\gamma}(\phi) Y^{\beta'}(f)Y^{\alpha}(\varphi). $$
	Since $|\beta|\leq N-n-1$, the signatures of $P_{d\gamma\beta'}^{\beta,i}(\varphi)$ are less than $N-n-2$. By bootstrap assumptions, we have
	$$|P_{d\gamma\beta'}^{\beta,i}(\varphi)|\lesssim (1+\log(1+t))^{N} \lesssim (1+t)^{\sigma_0},$$
	$$|\partial_{x^i}Z^{\gamma}(\phi)|\lesssim \frac{\epsilon^{\frac{1}{2}}}{(1+t)^2},$$
	where $\sigma_0\in(0,1)$ is a small number to be fixed later. Therefore, we have
	\begin{equation}
	\|I_1\|_{L^1(\mathbb{R}^n_x\times\mathbb{R}^n_v)}\lesssim \frac{\epsilon^{\frac{1}{2}}}{(1+t)^{2-\sigma_0}}\mathcal{F}(t).
	\end{equation}
	
	{\bf Estimate for $\|I_2\|_{L^1(\mathbb{R}^n_x\times\mathbb{R}^n_v)}$.} By \eqref{eq.mv.tphi}, we have
	$$I_2=\sum_{d=0}^{|\alpha|+1}\sum_{i=1}^{n}\sum_{|\gamma|,|\beta'|\leq |\alpha|} P_{d\gamma\beta'}^{\alpha,i}(\varphi)\partial_{x^i}Z^{\gamma}(\phi) Y^{\beta'}(\varphi)Y^{\beta}(f).$$
	Here the multi-linear form $P_{d\gamma\beta'}^{\beta,i}(\varphi)$ has signature less than $k\leq |\alpha|-1$ and $k+|\gamma| + |\beta'| \leq |\alpha| +1\leq N$. Now if $|\gamma|\leq N-n-1$, then
	$$|\partial_{x^i}Z^{\gamma}(\phi)|\lesssim \frac{\epsilon^{\frac{1}{2}}}{(1+t)^2}.$$
	$P_{d\gamma\beta'}^{\beta,i}(\varphi)Y^{\beta'}(\varphi)$ is a multi-linear form with at most one factor $Y^{\alpha'}(\varphi)$ with $N-n-2<|\alpha'|\leq |\alpha|$, while the rest can be uniformly bounded by $(1+\log(1+t))^N\lesssim (1+t)^{\sigma_0}$. Therefore, we have
	$$\|P_{d\gamma\beta'}^{\alpha,i}(\varphi)\partial_{x^i}Z^{\gamma}(\phi) Y^{\beta'}(\varphi)Y^{\beta}(f)\|_{L^1(\mathbb{R}^n_x\times\mathbb{R}^n_v)}\lesssim \frac{\epsilon^{\frac{1}{2}}}{(1+t)^{2-\sigma_0}}\mathcal{F}(t).$$
	If $|\gamma|>N-n-1$, then the bootstrap assumptions tell us
	$$|P_{d\gamma\beta'}^{\beta,i}(\varphi)Y^{\beta'}(\varphi)|\lesssim (1+\log(1+t))^{N} ,$$
	Since $|\gamma|\leq |\alpha|\leq N-1$, by \lmref{lm.mv.grad.term}, we have
	$$\|\partial_{x^i} Z^{\gamma}(\phi) Y^{\beta} (f)\|_{L^1(\mathbb{R}^n_x\times\mathbb{R}^n_v)}\lesssim \sum_{|\eta|\leq |\gamma|}\frac{\epsilon}{(1+t)^2}\|Z^{\eta}f\|_{L^1(\mathbb{R}^n_x\times\mathbb{R}^n_v)} .$$
	By \lmref{lm.mv.Z.Y}, we have
	$$\|Z^{\eta}f\|_{L^1(\mathbb{R}^n_x\times\mathbb{R}^n_v)}\leq \sum_{d'=0}^{|\eta|}\sum_{|\eta'|\leq |\eta|}\|P_{d'\eta'}^{\eta}(\varphi)Y^{\eta'}(f)\|_{L^{1}(\mathbb{R}^n_x\times\mathbb{R}^n_v)}\lesssim (1+\log(1+t))^{N}\mathcal{F}(t).$$
	So we have
	$$\|P_{d\gamma\beta'}^{\alpha,i}(\varphi)\partial_{x^i}Z^{\gamma}(\phi) Y^{\beta'}(\varphi)Y^{\beta}(f)\|_{L^1(\mathbb{R}^n_x\times\mathbb{R}^n_v)}\lesssim \frac{\epsilon(1+\log(1+t))^{2N}}{(1+t)^{2}}\mathcal{F}(t)\lesssim \frac{\epsilon}{(1+t)^{2-\sigma_0}}\mathcal{F}(t).$$
	In summary, we have 
	\begin{equation}
	\|I_2\|_{L^1(\mathbb{R}^n_x\times\mathbb{R}^n_v)}\lesssim \frac{\epsilon^{\frac{1}{2}}}{(1+t)^{2-\sigma_0}}\mathcal{F}(t).
	\end{equation}
	
	{\bf Estimate for $\|I_3\|_{L^1(\mathbb{R}^n_x\times\mathbb{R}^n_v)}$.}  First, for the modified vector field, from \eqref{eq.modi.phi.eq}, we know that
	$$T_{\phi}(\varphi)= t\sum_{i=1}^{n}\sum_{|\eta|\leq 1}c_{Z,i}\partial_{x^i}Z^{\eta}\phi.$$
	By \lmref{lm.mv.Y.na}, we have
	\begin{align*}
	I_3 & =t\sum_{|\eta|\leq |\alpha|+1}c_{\eta,i}\partial_{x^i}Z^{\eta}(\phi)Y^{\beta}(f) +\sum_{d=1}^{|\alpha|}\sum_{|\eta|\leq |\alpha| +1} P_{d\eta}^{\alpha}(\varphi)\partial_{x^i}Z^{\eta}(\phi)Y^{\beta}(f)\\
	&=I_{3,1}+I_{3,2},
	\end{align*}
	where $P_{d\eta}^{\alpha}(\varphi)$ are multi-linear forms of degree $d$ with signatures less than $k$ satisfying $k\leq |\alpha|\leq N-1$ and $k+|\eta|\leq |\alpha|+1$. Now, let's look at $I_{3,2}$. When $|\eta|\leq N-n-1$, we have 
	$$|\partial_{x^i}Z^{\eta}(\phi)|\lesssim \frac{\epsilon^{\frac{1}{2}}}{(1+t)^2},$$
	$$\|P_{d\eta}^{\alpha}(\varphi)Y^{\beta}(f)\|_{L^1(\mathbb{R}^n_x\times\mathbb{R}^n_v)}\lesssim (1+\log(1+t))^N\mathcal{F}(t).$$
	So we have,
	$$\|P_{d\eta}^{\alpha}(\varphi)\partial_{x^i}Z^{\eta}(\phi)Y^{\beta}(f)\|_{L^1(\mathbb{R}^n_x\times\mathbb{R}^n_v)}\lesssim  \frac{\epsilon^{\frac{1}{2}}}{(1+t)^{2-\sigma_0}}\mathcal{F}(t).$$
	When $|\eta|>N-n-1$, we have
	$$|P_{d\eta}^{\alpha}(\varphi)|\lesssim (1+\log(1+t))^N.$$
	Therefore, we have
	$$\|I_3\|_{L^1(\mathbb{R}^n_x\times\mathbb{R}^n_v)}\lesssim (t+1)\sum_{|\eta|\leq |\alpha| +1} \|\partial_{x^i}Z^{\eta}(\phi)Y^{\beta}(f)\|_{L^1(\mathbb{R}^n_x\times\mathbb{R}^n_v)} + \frac{\epsilon^{\frac{1}{2}}}{(1+t)^{2-\sigma_0}}\mathcal{F}(t).$$
	Now we only need to estimate $\|\partial_{x^i} Z^{\eta}(\phi) Y^{\beta} (f)\|_{L^1(\mathbb{R}^n_x\times\mathbb{R}^n_v)}$. By \lmref{lm.mv.grad.term}, we have
	\begin{align*}
	\|\partial_{x^i} Z^{\eta}(\phi) Y^{\beta} (f)\|_{L^1(\mathbb{R}^n_x\times\mathbb{R}^n_v)} & \lesssim \sum_{|\eta'|\leq |\eta|}\frac{\epsilon}{(1+t)^2} \|\rho(Z^{\eta'}f)\|_{L^1(\mathbb{R}^n_x)}\\
	& \lesssim \frac{\epsilon^2}{(1+t)^2} + \sum_{1\leq|\eta'|\leq |\eta|}\frac{\epsilon}{(1+t)^2} \|\rho(Z^{\eta'}f)\|_{L^1(\mathbb{R}^n_x)}.
	\end{align*}
	Now, for $|\eta'|\geq 1$, we can write $Z^{\eta'}f=Z^{\eta''}(Zf)$, where $0\leq|\eta''|= |\eta'|-1\leq N-1$. We apply \lmref{lm.mv.rho.Z} to $Zf$, then we have
	\begin{align*}
	\rho(Z^{\eta'} f) & =\sum_{d=0}^{|\eta''|}\sum_{|\beta'|\leq |\eta''|}\rho(Q_{d\beta'}^{\eta''}(\partial_x \varphi)Y^{\beta'}(Z f)) +\sum_{j=1}^{|\eta''|}\sum_{d=1}^{|\eta''|+1}\sum_{|\beta'|\leq |\eta|}\frac{1}{t^j}\rho(P_{d\beta'}^{\eta'',j}(\varphi) Y^{\beta'}( Zf))\\
	 & =P_1 +P_2
	\end{align*}
	where $Q_{d\beta'}^{\eta''}(\partial_x \varphi)$ are multi-linear forms with respect to $\partial_x \varphi$ of signatures less than $k'$ satisfying $k'\leq |\eta''|-1\leq N-2$ and $k' + d +|\beta'|\leq |\eta''|$, $P_{d\beta'}^{\eta'',j}(\varphi)$ are multilinear froms of degree $d$ with signatures less than $k$ satisfying $k\leq |\eta''| $ and $k+|\beta'|\leq |\eta''|\leq N-1$. For $P_1$, we have
	\begin{align*}
	\rho(Q_{d\beta'}^{\eta''}(\partial_x \varphi)Y^{\beta'}(Z f)) & =\rho\left(Q_{d\beta'}^{\eta''}(\partial_x \varphi)Y^{\beta'}\big(Y(f)-c_Y\varphi \partial_ {x^j}(f)\big)\right) \\
	& =\rho(Q_{d\beta'}^{\eta''}(\partial_x \varphi)(Y^{\beta'}Y(f))-\sum_{|\beta''|\leq |\beta'|}c_{Y\beta''}\rho(Q_{d\beta'}^{\eta''}(\partial_x \varphi)Y^{\beta''}(\varphi) Y^{\beta'-\beta''}\partial_ {x^j}(f))).
	\end{align*}
	Since $k'\leq N-2$ and $k'+d +|\beta'|+1\leq |\eta''| +1<N+1$, we have
	$$\|\rho(Q_{d\beta'}^{\eta''}(\partial_x \varphi)(Y^{\beta'}Y(f))\|_{L^1(\mathbb{R}^n_x)}\lesssim \mathcal{G}(t).$$
	For the second part of $P_1$, we have either $k'+d \leq N-n-2$ or $|\beta''|\leq N-n-2$, by bootstrap assumptions, we have
	$$\|\rho(Q_{d\beta'}^{\eta''}(\partial_x \varphi)Y^{\beta''}(\varphi) Y^{\beta'-\beta''}\partial_ {x^j}(f)))\|_{L^1(\mathbb{R}^n_x)}\lesssim \mathcal{F}(t) + (1+\log(1+t))\mathcal{G}(t).$$
	Therefore,
	$$\|P_1\|_{L^1(\mathbb{R}^n_x)} \lesssim \mathcal{F}(t) + (1+t)^{\sigma_0}\mathcal{G}(t).$$
	For $P_2$, by using the form $Z=Y+\varphi\partial_x$, we have
	$$\|P_2\|_{L^1(\mathbb{R}^n_x)}\lesssim \frac{(1+\log(1+t))^N}{t}\mathcal{F}(t).$$
	In summary, we have
	$$\|I_3\|_{L^1(\mathbb{R}^n_x\times\mathbb{R}^n_v)}\lesssim \frac{\epsilon^2}{1+t} + \frac{\epsilon}{1+t}\mathcal{F}(t) + \frac{\epsilon}{(1+t)^{1-\sigma_0}}\mathcal{G}(t) + \frac{\epsilon^{\frac{1}{2}}}{(1+t)^{2-\sigma_0}}\mathcal{F}(t).$$
	
	Now, if $Y^{\alpha}=Y^{\alpha'}\partial_{x^j}$, then $I_3$ becomes
	$$I_3=t\sum_{|\eta|\leq |\alpha|}c_{\eta,il}\partial_{x^i}\partial_{x^l}Z^{\eta}(\phi)Y^{\beta}(f) +\sum_{d=1}^{|\alpha|-1}\sum_{|\eta|\leq |\alpha|} P_{d\eta}^{\alpha,il}(\varphi)\partial_{x^i}\partial_{x^l}Z^{\eta}(\phi)Y^{\beta}(f).$$
	Since $t\partial_{x^l}\in \Gamma$, then the estimate for $I_3$ can be improved by
	$$\|I_3\|_{L^1(\mathbb{R}^n_x\times\mathbb{R}^n_v)}\lesssim \sum_{|\eta|\leq |\alpha| +1} \|\partial_{x^i}Z^{\eta}(\phi)Y^{\beta}(f)\|_{L^1(\mathbb{R}^n_x\times\mathbb{R}^n_v)} + \frac{\epsilon^{\frac{1}{2}}}{(1+t)^{2-\sigma_0}}\mathcal{F}(t),$$
	Therefore, we have the improved estimates
	$$\|I_3\|_{L^1(\mathbb{R}^n_x\times\mathbb{R}^n_v)}\lesssim \frac{\epsilon^2}{(1+t)^2} + \frac{\epsilon}{(1+t)^{2-\sigma_0}}\mathcal{G}(t) + \frac{\epsilon^{\frac{1}{2}}}{(1+t)^{2-\sigma_0}}\mathcal{F}(t).$$
	
	In summary,  for $|\alpha|>N-n-2$, we have
	$$\|T_{\phi}(Y^{\alpha}(\varphi)Y^{\beta}(f))\|_{L^1(\mathbb{R}^n_x\times\mathbb{R}^n_v)}\lesssim \frac{\epsilon^2}{1+t} + \frac{\epsilon}{1+t}\mathcal{F}(t) + \frac{\epsilon}{(1+t)^{1-\sigma_0}}\mathcal{G}(t) + \frac{\epsilon^{\frac{1}{2}}}{(1+t)^{2-\sigma_0}}\mathcal{F}(t).$$
	For $|\alpha|>N-n-3$, we have
	$$\|T_{\phi}(Y^{\alpha}(\partial_x\varphi)Y^{\beta}(f))\|_{L^1(\mathbb{R}^n_x\times\mathbb{R}^n_v)}\lesssim \frac{\epsilon^2}{(1+t)^2} + \frac{\epsilon}{(1+t)^{2-\sigma_0}}\mathcal{G}(t) + \frac{\epsilon^{\frac{1}{2}}}{(1+t)^{2-\sigma_0}}\mathcal{F}(t).$$
	As a consequence, we have
	\begin{align*}
	\|Y^{\alpha}(\varphi(t))Y^{\beta}(f(t))\|_{L^1(\mathbb{R}^n_x\times\mathbb{R}^n_v)} & \lesssim \int_{0}^{t}\|T_{\phi}(Y^{\alpha}(\varphi(s))Y^{\beta}(f(s)))\|_{L^1(\mathbb{R}^n_x\times\mathbb{R}^n_v)}{\rm d}s\\
	& \lesssim \epsilon^2\log(1+t) + \epsilon^{\frac{1}{2}}\int_{0}^{t}\frac{\mathcal{F}(s)}{1+s}{\rm d }s +\epsilon \int_{0}^{t}\frac{\mathcal{G}(s)}{(1+s)^{1-\sigma_0}}{\rm d }s.
	\end{align*}
	\begin{align*}
	\|Y^{\alpha}(\partial_x\varphi(t))Y^{\beta}(f(t))\|_{L^1(\mathbb{R}^n_x\times\mathbb{R}^n_v)} & \lesssim \int_{0}^{t}\|T_{\phi}(Y^{\alpha}(\partial_x\varphi(s))Y^{\beta}(f(s)))\|_{L^1(\mathbb{R}^n_x\times\mathbb{R}^n_v)}{\rm d}s\\
	& \lesssim \epsilon^2+ \epsilon^{\frac{1}{2}}\int_{0}^{t}\frac{\mathcal{F}(s)}{(1+s)^{2-\sigma_0}}{\rm d }s +\epsilon \int_{0}^{t}\frac{\mathcal{G}(s)}{(1+s)^{2-\sigma_0}}{\rm d }s
	\end{align*}
	
	In summary, we have
	\begin{equation}\label{eq.mv.est.F}
	\mathcal{F}(t)\lesssim \epsilon^{\frac{3}{2}}(1+t)^{\sigma_0} + \epsilon^{\frac{1}{2}}\int_{0}^{t}\frac{\mathcal{F}(s)}{1+s}{\rm d }s +\epsilon \int_{0}^{t}\frac{\mathcal{G}(s)}{(1+s)^{1-\sigma_0}}{\rm d }s,
	\end{equation}
	\begin{equation}\label{eq.mv.est.G}
	\mathcal{G}(t)\lesssim \epsilon + \epsilon^{\frac{1}{2}}\int_{0}^{t}\frac{\mathcal{F}(s)}{(1+s)^{2-\sigma_0}}{\rm d }s +\epsilon \int_{0}^{t}\frac{\mathcal{G}(s)}{(1+s)^{2-\sigma_0}}{\rm d }s.
	\end{equation}
	We apply Gronwall's lemma to \eqref{eq.mv.est.F}, then we have
	\begin{equation}
	\mathcal{F}(t)\lesssim\left(\epsilon^{\frac{3}{2}}(1+t)^{\sigma_0}  + \epsilon \int_{0}^{t}\frac{\mathcal{G}(s)}{(1+s)^{1-\sigma_0}}{\rm d }s\right) (1+t)^{C\epsilon^{\frac{1}{2}}}.
	\end{equation}
	We apply this to \eqref{eq.mv.est.G}, then we have
	\begin{align*}
	\mathcal{G}(t) & \lesssim \epsilon  +\epsilon \int_{0}^{t}\frac{\mathcal{G}(s)}{(1+s)^{2-\sigma_0}}{\rm d }s + \epsilon^2\int_{0}^t\frac{1}{(1+s)^{2-2\sigma_0-C\epsilon^{\frac{1}{2}}}}{\rm d}s\\
	&\quad + \epsilon^{\frac{3}{2}}\int_0^t \frac{{\rm d}s}{(1+s)^{2-\sigma_0-C\epsilon^{\frac{1}{2}}}}\int_0^s\frac{\mathcal{G}(\tau)}{(1+\tau)^{1-\sigma_0}}{\rm d }\tau.
	\end{align*}
	The last term
	\begin{align*}
	{\rm Last} & =\epsilon^{\frac{3}{2}}\int_0^t \frac{\mathcal{G}(\tau)}{(1+\tau)^{1-\sigma_0}}{\rm d }\tau\int_\tau^t
	\frac{{\rm d}s}{(1+s)^{2-\sigma_0-C\epsilon^{\frac{1}{2}}}}\\
	& \leq \frac{\epsilon^{\frac{3}{2}}}{1-\sigma_0-C\epsilon^{\frac{1}{2}}} \int_0^t \frac{\mathcal{G}(\tau)}{(1+\tau)^{2-2\sigma_0-C\epsilon^{\frac{1}{2}}}}{\rm d }\tau.
	\end{align*}
	Now we choose $\sigma_0$ and $\epsilon_\sigma$ such that $2\sigma_0 +C\epsilon_\sigma^{\frac{1}{2}}\leq \min\{\frac{1}{2},\sigma\}$, then we have
	$$\mathcal{G}(t)\lesssim \epsilon,\quad \mathcal{F}(t)\lesssim \epsilon(1+t)^{\sigma}.$$

\end{proof}

\subsection{Improving the Bootstrap Assumptions}	

\begin{lemma}
	If the initial data $f_0$ satisfies $\mathcal{E}_N[f_0]\leq \epsilon$, then when $\epsilon$ is small enough, we have that, for all $t\in [0,T]$, 
	$$\mathcal{E}_N[f(t)]\leq \frac{3}{2}\epsilon.$$
\end{lemma}

\begin{proof}
	For any multi-index $\alpha\leq N$, we have
	$$[T_{\phi}, Y^{\alpha}](f)=\sum_{d=0}^{|\alpha|+1}\sum_{i=1}^{n}\sum_{|\gamma|,|\beta|\leq |\alpha|} P_{d\gamma\beta}^{\alpha,i}(\varphi)\partial_{x^i}Z^{\gamma}(\phi) Y^{\beta}(f),$$
	where $P_{d\gamma\beta}^{\alpha,i}(\varphi)$ are multi-linear forms of degree $d$ with signatures less than $k$ satisfying that $k\leq |\alpha|-1$ and $k+|\gamma|+|\beta|\leq |\alpha| +1\leq N+1$.  We have two different cases. When $|\gamma|\leq N-n-1$, then we have
	$$|\partial_{x^i}Z^{\gamma}(\phi)|\leq \frac{\epsilon^{\frac{1}{2}}}{(1+t)^2}.$$
	Since $k+|\beta|\leq N+1$ and $k\leq N-1$, by \lmref{lm.mv.YY}, we have,
	$$\|P_{d\gamma\beta}^{\alpha,i}(\varphi) Y^{\beta}(f)\|_{L^1(\mathbb{R}^n_x\times\mathbb{R}^n_v)}\lesssim (1+\log(t+1))^{N+1}(1+t)^{\sigma}\epsilon.$$
	By taking $\sigma$ small enough, we have that
	$$\|P_{d\gamma\beta}^{\alpha,i}(\varphi)\partial_{x^i}Z^{\gamma}(\phi) Y^{\beta}(f)\|_{L^1(\mathbb{R}^n_x\times\mathbb{R}^n_v)}\lesssim  \frac{\epsilon^{\frac{3}{2}}}{(1+t)^{1+\sigma'}},$$
	where $\sigma'>0$.
	When $|\gamma|>N-n-1$, then $k, |\beta|\leq N-n-2$ since $N\geq 2n+3$, so we have,
	$$|P_{d\gamma\beta}^{\alpha,i}(\varphi)|\lesssim (1+\log(1+t))^{N+1}.$$
	By \lmref{lm.mv.grad.term}, we have
	$$\|\partial_{x^i}Z^{\gamma}(\phi) Y^{\beta}(f)\|_{L^1(\mathbb{R}^n_x\times\mathbb{R}^n_v)}\lesssim  \frac{\epsilon}{(1+t)^2}\sum_{|\eta|\leq |\gamma|}\|Z^{\eta}(f)\|_{L^1(\mathbb{R}^n_x\times\mathbb{R}^n_v)}.$$
	By \lmref{lm.mv.Z.Y}, we have
	$$Z^{\eta}(f)=\sum_{d=0}^{|\eta|}\sum_{|\eta'|\leq |\eta|}P_{d\eta'}^{\eta}(\varphi)Y^{\eta'}(f),$$
	where $P_{d\eta'}^{\eta}(\varphi)$ are multi-linear forms of signature less than $k$ with $k\leq |\eta|-1\leq N-1$ and $k+|\eta'|\leq |\eta|\leq N$. Therefore, by \lmref{lm.mv.YY}, we have
	$$\|Z^{\eta}(f)\|_{L^1(\mathbb{R}^n_x\times\mathbb{R}^n_v)}\lesssim (1+\log(t+1))^{N}(1+t)^{\sigma}\epsilon,$$
	which implies that there exist $\sigma'>0$ such that
	$$\|P_{d\gamma\beta}^{\alpha,i}(\varphi)\partial_{x^i}Z^{\gamma}(\phi) Y^{\beta}(f)\|_{L^1(\mathbb{R}^n_x\times\mathbb{R}^n_v)}\lesssim  \frac{\epsilon^2}{(1+t)^{1+\sigma'}}.$$
	In summary, we have that, there exists $\sigma'>0$, such that
	$$\|T_{\phi}Y^{\alpha}(f)\|_{L^1(\mathbb{R}^n_x\times\mathbb{R}^n_v)}\lesssim \frac{\epsilon^{\frac{3}{2}}}{(1+t)^{1+\sigma'}}.$$
	Therefore, when $\epsilon$ is small enough,
	\begin{align*}
	\mathcal{E}_{N}[f(t)] & \leq \mathcal{E}_{N}[f_0] +\sum_{|\alpha|\leq N}\int_0^{t}\|T_{\phi}Y^{\alpha}(f)\|_{L^1(\mathbb{R}^n_x\times\mathbb{R}^n_v)}\\
	& \leq \epsilon + C\epsilon^{\frac{3}{2}}\int_0^{\infty}\frac{1}{(1+s)^{1+\sigma'}}{\rm d}s\leq \frac{3}{2}\epsilon.
	\end{align*}

\end{proof}

\begin{lemma}
	For all multi-index $\alpha$ with $|\alpha|\leq N-n-1$, we have
	\begin{equation}
	|\nabla_xZ^{\gamma}\phi(t,x)|\lesssim \frac{\epsilon}{(1+t)^2}.
	\end{equation}         
\end{lemma}

\begin{proof}
	The proof is exactly the same as \lmref{lm.vp.phiZ} and \lmref{lm.vy.phiZ}. For the V-P system in $n=3$, we have that $Z^{\gamma}\phi$ can be written as
	$$\nabla_xZ^{\gamma}\phi(x) = \sum_{|\alpha|\leq |\gamma|}\int_{\mathbb{R}^3}\frac{c_\alpha}{|y|^2}\frac{y}{|y|}\rho(Z^{\alpha}f)(x-y){\rm d}y,$$
	where $c_\alpha$ are some globally bounded constants. Then we have,
	$$|\nabla_xZ^{\gamma}\phi(x) |\lesssim\sum_{|\alpha|\leq |\gamma|} \int \frac{1}{|y|^2}|\rho(Z^{\alpha}f)(x-y)|{\rm d}y.$$
	By \lmref{lm.mv.rho.Z}, we have
	\begin{equation}
	\rho(Z^{\alpha} f) =\sum_{d=0}^{|\alpha|}\sum_{|\beta|\leq |\alpha|}\rho(Q_{d\beta}^{\alpha}(\partial_x \varphi)Y^{\beta} f) +\sum_{j=1}^{|\alpha|}\sum_{d=1}^{|\alpha|+1}\sum_{|\beta|\leq |\alpha|}\frac{1}{t^j}\rho(P_{d\beta}^{\alpha,j}(\varphi) Y^{\beta}(f)),
	\end{equation}
	where $Q_{d\beta}^{\alpha}(\partial_x \varphi)$ are multi-linear forms with respect to $\partial_x \varphi$ of signatures less than $k'$ satisfying $k'\leq |\alpha|-1\leq N-n-2$ and $k' + d +|\beta|\leq |\alpha|$, $P_{d\beta}^{\alpha,j}(\varphi)$ are multi-linear forms of degree $d$ with signatures less than $k$ satisfying $k\leq |\alpha| \leq N-n-1 $ and $k+|\beta|\leq |\alpha|$. Now we apply the Klainerman-Sobolev inequality to every term in the above equation, then we have,
	$$|\rho(Q_{d\beta}^{\alpha}(\partial_x \varphi)Y^{\beta} f)(x-y)|\lesssim \frac{1}{(1+t+|x-y|)^{n}} \sum_{|\eta|\leq n}\|Y^{\eta}[Q_{d\beta}^{\alpha}(\partial_x \varphi)Y^{\beta} f]\|_{L^1(\mathbb{R}^n_x\times\mathbb{R}^n_x)},$$
	$$|\rho(P_{d\beta}^{\alpha,j}(\varphi) Y^{\beta}(f))(x-y)|\lesssim \frac{1}{(1+t+|x-y|)^{n}} \sum_{|\eta|\leq n}\|Y^{\eta}[P_{d\beta}^{\alpha,j}(\varphi) Y^{\beta}(f)]\|_{L^1(\mathbb{R}^n_x\times\mathbb{R}^n_x)}.$$
	Now, since $N\geq 2n+3$, there is at most one multiplier $Y^{\eta'}(\varphi)$ with $|\eta'|>N-n-2$, therefore, by the bootstrap assumption and \lmref{lm.mv.YY}, we have 
	$$|\rho(Z^{\alpha}f)(x-y)|\lesssim \frac{\epsilon}{(1+t+|x-y|)^n} +\frac{\epsilon(\log(1+t)+1)^N}{(1+t+|x-y|)^nt}\lesssim \frac{\epsilon}{(1+t+|x-y|)^n}.$$ 
	By \lmref{lm.decay}, we have 
	$$|\nabla_xZ^{\gamma}\phi(t,x)|\lesssim \frac{\epsilon}{(1+t)^{n-1}}=\frac{\epsilon}{(1+t)^{2}}, \quad n=3.$$
	For the V-Y system, we can do the exactly similar estimate by using the exact expression of $\nabla_x Z^{\gamma} \phi$.

\end{proof}

\begin{lemma}
	For all multi-index $\alpha$ with $|\alpha|\leq N-n-2$, we have
	\begin{equation}
	|Y^{\alpha}\varphi(t,x,v)|\lesssim \epsilon (1+\log(1+t)).
	\end{equation}
	For all multi-index $\alpha$ with $|\alpha|\leq N-n-3$, we have
	\begin{equation}
	|Y^{\alpha}\nabla_x\varphi(t,x,v)|\lesssim \epsilon.
	\end{equation}
\end{lemma}

\begin{proof}
	By method of characteristics, we have that
	$$|Y^{\alpha}\varphi(t,x,v)|\leq \int_0^t \|T_\phi Y^{\alpha}(\varphi)(s)\|_{L^\infty(\mathbb{R}^n_x\times\mathbb{R}^n_v)}{\rm d}s.$$
	We just need to estimate $\|T_{\phi}Y^{\alpha}(\varphi)\|_{L^\infty}$. We have that,
	$$T_{\phi}Y^{\alpha}(\varphi)=Y^{\alpha}T_{\phi}(\varphi) + [T_{\phi},Y^{\alpha}](\varphi).
	$$
	For the first part $Y^{\alpha}T_{\phi}(\varphi)$, from \eqref{eq.modi.phi.eq}, we know that
	$$T_{\phi}(\varphi)= t\sum_{i=1}^{n}\sum_{|\eta|\leq 1}c_{Z,i}\partial_{x^i}Z^{\eta}\phi.$$
	Therefore, by \lmref{lm.mv.Y.na} and \lmref{lm.commu.ZZ}, we have 
	$$Y^{\alpha}T_{\phi}(\varphi)=t\sum_{|\eta|\leq |\alpha|+1}c_{\eta,i}\partial_{x^i}Z^{\eta}(\phi) +\sum_{d=1}^{|\alpha|}\sum_{|\eta|\leq |\alpha| +1} P_{d\eta}^{\alpha}(\varphi)\partial_{x^i}Z^{\eta}(\phi),$$
	where $P_{d\eta}^{\alpha}(\varphi)$ are multi-linear forms of degree $d$ with signatures less than $k$ satisfying $k\leq |\alpha|\leq N-n-2$ and $k+|\eta|\leq |\alpha|+1\leq N-n-1$. Then by bootstrap assumptions and the improved estimates for $\partial_{x^i}Z^{\eta}(\phi)$, we have
	$$|Y^{\alpha}T_{\phi}(\varphi)(t)|\lesssim \frac{\epsilon[t+(1+\log (1+t))^{N+1}]}{(1+t)^2}\lesssim \frac{\epsilon}{1+t}.$$
	For the second part $[T_{\phi},Y^{\alpha}](\varphi)$, by \eqref{eq.mv.tphi}, we have
	$$[T_{\phi},Y^{\alpha}](\varphi)=\sum_{d=0}^{|\alpha|+1}\sum_{i=1}^{n}\sum_{|\gamma|,|\beta'|\leq |\alpha|} P_{d\gamma\beta'}^{\alpha,i}(\varphi)\partial_{x^i}Z^{\gamma}(\phi) Y^{\beta'}(\varphi).$$
	Here the multi-linear form $P_{d\gamma\beta'}^{\beta,i}(\varphi)$ has signature less than $k\leq |\alpha|-1\leq N-n-3$ and $k+|\gamma| + |\beta'| \leq |\alpha| +1\leq N -n-1$. Therefore, by bootstrap assumptions, we have
	$$|[T_{\phi},Y^{\alpha}](\varphi)|\lesssim \frac{\epsilon(1+\log (1+t))^{N+1}}{(1+t)^2}\lesssim \frac{\epsilon}{1+t}.$$
	In summary, we have
	$$|Y^{\alpha}\varphi(t,x,v)|\lesssim \int_0^t \frac{\epsilon}{1+s}{\rm d}s \lesssim \epsilon (1+\log (1+t)).$$
	Now if $Y^{\alpha}=Y^{\alpha}\partial_x$, in all the above estimates, the term $\partial_{x^i}Z^{\eta}(\phi)$ will in fact be $\partial_{x^i}\partial_{x^j}Z^{\eta'}(\phi)$, which will provide an additional decay power in $t$ since $\partial_{x^i}\partial_{x^j}Z^{\eta'}(\phi)=t^{-1}\partial_{x^i}(t\partial_{x^j})Z^{\eta'}(\phi)$. So we have,
	$$|Y^{\alpha}\nabla_x\varphi(t,x,v)|\lesssim \int_0^t\frac{\epsilon}{(1+s)^2}{\rm d}s\lesssim \epsilon.$$
\end{proof}	

In summary, we improve the bootstrap assumptions \eqref{eq.bootstrap.1}-\eqref{eq.bootstrap.4} and therefore end the proof of Theorem \ref{thm.low}. 



\section*{Acknowledgments}
This research is funded by the European Research Council under the European Union’s Horizon 2020 research and innovation program (project GEOWAKI, grant agreement 714408). The author would like to thank Jacques Smulevici for many inspiring discussions during the completion of this paper.



\begin{thebibliography}{1}
\addcontentsline{toc}{section}{References}

\bibitem{Bardos.Degond} C. Bardos, P. Degond
{\it Global existence for the Vlasov-Poisson equation in 3 space variables with small initial data,}
{Ann. Inst. H. Poincar\'e Anal. Non Lin\'eaire 2 (1985), no. 2, 101–118. }

\bibitem{Bigorgne.maxwell.high} L. Bigorgne
{\it Asymptotic properties of small data solutions of the Vlasov-Maxwell system in high dimensions,}
{arXiv:1712.09698, 2017. }

\bibitem{Bigorgne.maxwell.maxwell.3d} L. Bigorgne
{\it Sharp asymptotic behavior of solutions of the 3d vlasov-maxwell system with small data,}
{arXiv:1812.11897, 2018. }

\bibitem{Bigorgne.maxwell.massless} L. Bigorgne
{\it Sharp asymptotics for the solutions of the three-dimensional massless Vlasov-Maxwell system with small data,}
{arXiv:1812.09716, 2018. }

\bibitem{Bigorgne.maxwell.asy} L. Bigorgne
{\it Asymptotic properties of the solutions to the Vlasov-Maxwell system in the exterior of a light cone,}
{arXiv:1902.00764, 2019. }

\bibitem{Choi.Ha.Lee.Yukawa} S.-H. Choi, S.-Y. Ha, H. Lee,
{\it Dispersion estimates for the two-dimensional Vlasov-Yukawa system with small data,}
{J. Differential Equations 250 (2011), no. 1, 515–550.}

\bibitem{Fajman.Joudioux.Smulevici.vector-field} D. Fajman, J. Joudioux, J. Smulevici,
{\it A vector field method for relativistic transport equations with applications,}
{Anal. PDE 10 (2017), no. 7, 1539–1612.}

\bibitem{Fajman.Joudioux.Smulevici.vlasov-nordstrom} D. Fajman, J. Joudioux, J. Smulevici,
{\it Sharp asymptotics for small data solutions of the Vlasov-Nordstr\"om system in three dimensions,}
{arXiv:1704.05353, 2017.}

\bibitem{Fajman.Joudioux.Smulevici.Einstein-Vlasov} D. Fajman, J. Joudioux, J. Smulevici, 
{\it The Stability of the Minkowski space for the Einstein-Vlasov system,}
{arXiv:1707.06141, 2017.}

\bibitem{Glassey.book} R.T. Glassey,
{\it The Cauchy problem in kinetic theory,}
{Society for Industrial and Applied Mathematics (SIAM), Philadelphia, PA, 1996. xii+241 pp.}

\bibitem{Hwang.Rendall.Velazquez} H. Hwang, A. Rendall, J. Vel\'azquez, 
{\it Optimal gradient estimates and asymptotic behaviour for the Vlasov-Poisson system with small initial data,}
{Arch. Ration. Mech. Anal. 200(1), 313–360 (2011).}

\bibitem{Klainerman} S. Klainerman,
{\it Uniform decay estimates and the Lorentz invariance of the classical wave equation,}
{Comm. Pure Appl. Math. 38 (1985), no. 3, 321–332.}

\bibitem{Lindblad.Taylor}H. Lindblad, M. Taylor,
{\it Global stability of Minkowski space for the Einstein-Vlasov system in the harmonic gauge,}
{arXiv:1707.06079, 2017.}

\bibitem{Lions.Perthame} P.-L. Lions, B. Perthame,
{\it Propagation of moments and regularity for the 3-dimensional Vlasov-Poisson system,}
{Invent. Math. 105(2), 415–430 (1991).}

\bibitem{Pfaffelmoser} K. Pfaffelmoser,
{\it Global classical solutions of the Vlasov-Poisson system in three dimensions for general initial data,}
{J. Differential Equations 95 (1992), no. 2, 281–303.}

\bibitem{Smulevici.V-P} J. Smulevici,
{\it Small data solutions of the Vlasov-Poisson system and the vector field method,}
{Ann. PDE 2 (2016), no. 2, Art. 11, 55 pp.}


\bibitem{Wang} X. Wang,
{\it Decay estimates for the 3D relativistic and non-relativistic Vlasov-Poisson systems,}
{arXiv:1805.10837, 2018.}

\bibitem{Bessel} G. N. Watson, 
{\it A treatise on the theory of Bessel functions,}
{2nd ed. reprint, Camb. Math. Lib., Cambridge Univ. Press, Cambridge, 1995.}

\bibitem{Yukawa} H. Yukawa, 
{\it On the interaction of elementary particles,}
{Proc. Phys. Math. Soc. Japan 17 (1935) 48–57.}


	
\end{thebibliography}
\end{document}